 \documentclass[final,3p,times]{elsarticle}

\usepackage{amssymb}
\usepackage{amsthm}

\usepackage{lineno}

\journal{Stochastic Processes and Their Applications}

\usepackage{mathtools}
\usepackage{xcolor}
\usepackage{enumitem}

\newtheorem{theorem}{Theorem}
\newtheorem{corollary}{Corollary}
\newtheorem{lemma}{Lemma} 
\newtheorem{assumption}{Assumption}
\newtheorem{assum}{Assumption}

\newdefinition{definition}{Definition}
\newdefinition{example}{Example}	

\DeclarePairedDelimiter\abs\lvert\rvert

\DeclarePairedDelimiter\floor{\lfloor}{\rfloor}

\newcommand{\NN}{\mathbb{N}}
\newcommand{\ZZ}{\mathbb{Z}}

\newcommand{\Zd}{{\ZZ^d}}
\newcommand{\R}{\mathbb{R}}
\newcommand{\Rd}{{\R^d}}
\newcommand{\vR}{{v\in \Rd}}
\newcommand{\PP}{\mathbb{P}}
\newcommand{\EE}{\mathbb{E}}
\newcommand{\FF}{\overline{F}}
\newcommand{\GG}{{\overline{G}}}
\newcommand{\dnm}{D_{n,k,L}^{-}}
\newcommand{\dnp}{D_{n,k,L}^{+}}
\newcommand{\ff}{f(v-u)}
\newcommand{\gga}{g^{\alpha}(\abs{v-u})}
\newcommand{\ffa}{f^{\alpha}(v-u)}
\newcommand{\fft}{f^{(t)}(v-u)}
\newcommand{\dd}{\mathrm{d}}
\newcommand{\zvt}{Z_v^{(t)}}
\newcommand{\myL}{M_{L,y}}
\newcommand{\myt}{M_{L,y}^{(t)}}
\newcommand{\taut}{\tau_L^{(t)}}
\newcommand{\bary}{\overline{y}(v)}
\newcommand{\supy}{y^*(v)}
\newcommand{\Fre}{\mathrm{MDA}(\Phi_{\alpha})}
\newcommand{\e}{\mathrm{e}}
\newcommand{\I}[1]{\boldsymbol{1}_{\{#1\}}}
\newcommand{\II}[1]{\boldsymbol{1}_{#1}}
\newcommand{\du}{\mathrm{d}u}
\newcommand{\Aa}{\mathcal{A}}
\newcommand{\Bb}{\mathcal{B}}

\newcommand{\RA}{\mathcal{R}_{-\alpha}}
\newcommand{\RV}{\mathcal{RV}_{\alpha}}
\newcommand{\ball}[1]{B(#1)}
\newcommand{\Tt}{t_{n,k,L}}
\newcommand{\tildet}{\tilde t_{n,k,L}}
\newcommand{\Iz}{I_z^{n,k,L}}
\newcommand{\Jz}{J_z^{n,k,L}}
\newcommand{\Pp}{P_{n,k,L}}
\newcommand{\pp}{p_{n,k,L}}
\newcommand{\Qq}{Q_{n,k,L}}
\newcommand{\qq}{q_{n,k,L}}
\newcommand{\Nk}{N_{k,L}}
\newcommand{\KL}{K_{n,L}}
\newcommand{\Kk}{\mathcal{K}_{n,L}}
\newcommand{\alphaL}{\alpha_{y,n,L}}

\begin{document}

\begin{frontmatter}



\title{Extremes of L\'evy-driven spatial random fields with regularly varying L\'evy measure}

\author[cbs]{Anders R{\o}nn-Nielsen}
\ead{aro.fi@cbs.dk}

\author[cbs]{Mads Stehr\corref{cor1}}
\ead{mast.fi@cbs.dk}
\cortext[cor1]{Corresponding author}

\affiliation[cbs]{
organization={Department of Finance, Copenhagen Business School},
            addressline={Solbjerg Plads~3}, 
            postcode={2000},
            city={Frederiksberg},             
            country={Denmark}}

\begin{abstract}
We consider an infinitely divisible random field indexed by $\mathbb{R}^d$, $d\in\mathbb{N}$, given as an integral of a kernel function with respect to a L\'evy basis with a L\'evy measure having a regularly varying right tail. First we show that the tail of its supremum over any bounded set is asymptotically equivalent to the right tail of the L\'evy measure times the integral of the kernel.
Secondly, when observing the field over an appropriately increasing sequence of continuous index sets, we obtain an extreme value theorem stating that the running supremum converges in distribution to the Fr\'echet distribution.
\end{abstract}




\end{frontmatter}


\noindent
{\it Declarations of interest:} None.\\[1ex]
\noindent
{\it Funding:} This research did not receive any specific grant from funding agencies in the public, commercial, or not-for-profit sectors. \\[5ex]
This work has been submitted to \emph{Stochastic Processes and their Applications}.

	\section{Introduction}

In this paper we consider the extremal behavior of a spatial L\'evy-driven moving average given by
\begin{equation}\label{eq:Xintro}
	X_v = \int_\Rd \ff \Lambda(\dd u),
	\qquad v\in \Rd,
\end{equation}
where $\Lambda$ is an infinitely divisible, independently scattered random measure on $\Rd$ (i.e. a L\'evy basis) and $f$ is an appropriate kernel function; see \cite[Theorem~2.7]{Rajput1989} for necessary and sufficient conditions guaranteeing its existence. L\'evy-driven models as in \eqref{eq:Xintro} form a very general modeling framework used for a wide range of purposes, including modeling of financial assets (\cite{Barndorff2001}), turbulent flows (\cite{Barndorff2004}), brain imaging data (\cite{Jonsdottir2013}) and wind power prices (\cite{Benth2018}). Estimators for its mean and variogram are suggested in \cite{RNielsen2019}, and central limit theorems for these estimators are also presented.

In this paper we assume that the L\'evy measure $\rho$ of the basis $\Lambda$
has a regularly varying right tail;
see for instance \cite[Appendix~A.3]{Embrecths1997}. The regularly varying distributions are in particular subexponential, and the class covers many interesting heavy-tailed distributions such as the Pareto, Cauchy, Loggamma, stable (of index $<2$) and, in particular, Fr\'echet distribution. Moreover (and essential to this paper), the set of regularly varying distributions coincide with the maximum domain of attraction of the Fr\'echet distribution; see \cite[Chapter~3]{Embrecths1997}.

In addition to assuming integrability of $f$ and regular variation of the right tail of $\rho$, we consider two different assumptions on the pair $(\Lambda,f)$.
In Assumption~\ref{ass:fullassumption1} we have minimal requirements on $\rho$ but we consider a H\"older continuous kernel $f$. In Assumption~\ref{ass:fullassumption2} we have minimal requirements on $f$, whereas we restrict consideration to L\'evy bases of finite variation, meaning in particular that the L\'evy measure $\rho$ has finite first moment in a neighborhood of $0$.

The first part of the paper concerns the asymptotic representation of the tail of $\sup_{v\in B} X_v$ for some fixed and bounded index set $B \subseteq \Rd$. With $\alpha>0$ denoting the regularly varying index of $\rho$ we show that
\begin{equation}\label{eq:supintro}
	\PP \bigl(\sup_{v\in B} X_v > x \bigr)
	\sim 
	\rho((x,\infty)) \int_\Rd \sup_{v\in B} \ffa \dd u
\end{equation}
as $x\to\infty$, where $\sim$ denotes asymptotic equivalence. In fact we show the more general result that the equivalence holds true even if we replace $X$ with $X + Y^1$, where $Y^1$ is an independent field with right tail lighter than $X$. To obtain the representation \eqref{eq:supintro} we use a result from \cite{RosinskiSamorodnitsky1993}, where we explicitly use that the right tail of $\rho$ is subexponential. 

Results similar to \eqref{eq:supintro} are found in the literature for one-dimensional regularly varying processes, and for spatial fields with lighter tails: In \cite{Fasen2005} a field
\begin{equation}\label{eq:fasenfield}
	X_t = \int_{\R_+ \times \R} f(r,t-s) \Lambda(\dd r, \dd s)
\end{equation}
is studied under the assumption of an underlying regularly varying L\'evy measure. Replacing $\Rd$ with $\R$ and $B$ with the fixed interval $[0,h]$, the claim \eqref{eq:supintro} is shown for the one-dimensional moving average. If, instead, the spatial L\'evy-driven field $X_v$ has a convolution equivalent L\'evy measure $\rho$ of strictly positive index (\cite{Cline1986,Cline1987,Pakes2004}), it is shown in \cite{RNielsen2016} that
\[
	\PP \bigl(\sup_{v\in B} X_v > x \bigr)
	\sim 
	\rho((x,\infty)) K \abs{B} ,
\] 
where $\abs{B}$ is Lebesgue measure of $B\subseteq \R^d$ and $K$ is a computable constant. A slightly more general result is given in \cite{Stehr2020b}, in which a space-time L\'evy model $X_{v,t}$ with a convolution equivalent L\'evy measure is considered. For a large class of functionals $\Psi$ acting on the field, it is shown that there are constants $c,C$ such that
\[
	\PP\bigl(\Psi(X_{v,t}) >x \bigr)
	\sim
	C \rho((x/c,\infty))
\]
as $x\to\infty$. For Gaussian random fields the distribution of the supremum can be approximated by the expected Euler characteristic of an excursion set (see \cite{Adler2007} and the references therein).

The second part of the paper concerns the asymptotic distribution of $\sup_{v\in C_n} X_v$ as $n\to\infty$, where $(C_n)$ is a sequence of index sets in $\Rd$ increasing appropriately. From extreme value theory for dependent stationary fields we know, assuming some mixing and anti-clustering conditions, that the distribution of the running maximum of a stationary field is determined by its marginal tail; see \cite{Embrecths1997,Leadbetter1983,Resnick2008} for detailed treatments of classical extreme value theory, and see \cite{Jakubowski2019,Soja2019,StehrRonnNielsen2020} for generalizations to stationary, discretely indexed $d$-dimensional spatial fields. In particular, if the marginal tail is in the maximum domain of attraction of the Fr\'echet distribution (or equivalently it is regularly varying), then the running maximum of the field converges to the Fr\'echet distribution. 
Thus, with \eqref{eq:supintro} in mind, we expect (and show) that the distribution of $\sup_{v\in C_n} X_v$ converges to the Fr\'echet distribution
if the index sets $C_n$ behave nicely:  
We require that $C_n$ is a union of a fixed number of connected convex bodies (convex, compact sets with non-empty interior) with intrinsic volumes sufficiently bounded relative to the volume of $C_n$; see \cite[Chapter~4]{Schneider1993} for an exposition of convex bodies and their intrinsic volumes. This includes the useful situation where a union of fixed convex bodies is scaled by an increasing real sequence.
With $\alpha>0$ denoting the regularly varying index of $\rho$ we then show that there are norming constants $(a_n)$ such that
\begin{equation}\label{eq:extremeintro}
	\PP \bigl(a_n^{-1} \sup_{v\in C_n} X_v > x \bigr)
	\to 
	\exp \bigl(-x^{-\alpha} \rho((1,\infty))\bigr)
\end{equation}
as $n\to\infty$. The result does not follow by the aforementioned papers on discretely indexed fields, as the appropriate mixing and anti-clustering conditions do not easily show for our continuously indexed field. Instead, the proof is based on a conditioning argument, where we apply an independent decomposition $X=Z + Y$ and condition on $Y=y$. Here $Z$ is a compound Poisson sum, which determines the supremum of $X$, and $Y$ represents the light-tailed and high-activity part of $X$.
 The limit \eqref{eq:extremeintro} is then obtained by first establishing the conditional result for $Z+y$ and then applying ergodic properties of the light-tailed field $Y$. This proof technique almost immediately implies the extended result that \eqref{eq:extremeintro} is also satisfied with $X$ replaced by $X + Y^1 + Y^2$, where the stationary and independent fields $Y^1$ and $Y^2$ satisfy that $Y^1$ is ergodic and that $X$ has heavier right tails than $Y^1$ and $\abs{Y^2}$.

Extremal results related to the one in the present paper are found in the literature in various forms. In \cite{Fasen2005} the running supremum $\sup_{t\in [0,T]} X_t$ of the one-dimensional moving average \eqref{eq:fasenfield} is shown to converge to the Fr\'echet distribution as $T\to\infty$. A slightly less general result is given in \cite{Rootzen1978}, in which the underlying L\'evy measure is assumed to be stable with index $<2$.

In \cite{StehrRonnNielsen2020} a $d$-dimensional L\'evy-driven field with a convolution equivalent L\'evy measure is considered under an asymptotic regime that, regarding the sequence of index sets, is identical to the present. Convolution equivalent distributions are in particular in the maximum domain of attraction of the Gumbel distribution, and thus, in this case, the running supremum $\sup_{v\in C_n} X_v$ converges to a power of the Gumbel distribution function $x\mapsto \exp(\e^{-x})$. The proof structure in \cite{StehrRonnNielsen2020} has some similarities with the proofs in the present paper, as both rely on conditioning on the light-tailed and heavy-activity part of the field. However, the obvious differences in both tail behavior and limits make the proofs substantially different.

The paper is organized as follows. In Section~\ref{sec:main} we formally define our L\'evy-driven field and provide assumptions on the basis, integration kernel and increasing index sets, before presenting the main results of the paper. 
In Section~\ref{sec:regvar} we give some useful results for regularly varying distributions, which are then used to show the tail representation of the supremum of the field in Section~\ref{sec:tail}. Section~\ref{sec:geometry} is devoted to geometric proofs related to the expansion of the index sets $(C_n)$, and proofs for the result on the running supremum $\sup_{v\in C_n} X_v$ are found in Section~\ref{sec:extreme}.

	\section{Definitions and main results}
	\label{sec:main}
In this section we formally define our random field and state sufficient assumptions after which we can present the main results.
Before doing so, we briefly clarify some notation used throughout the paper. We let $\abs{{}\cdot{}}$ denote size in the following sense: $\abs{v}$ is the Euclidean norm for a single (one- or multi-dimensional) point $v$, $\abs{A}$ is Lebesgue measure of a full-dimensional set $A\subseteq \Rd$, and $\abs{A}$ is the number of points in a discrete set $A\subseteq \Zd$. However, we will at times also use the notation $m$ for Lebesgue measure.

We consider a stationary L\'evy-driven random field $(X_v)_{v\in\Rd}$ given as a integral of a kernel function with respect to a L\'evy basis. A L\'evy basis is an infinitely divisible and independently scattered random measure. The random measure $\Lambda$ on $\Rd$ is independently scattered if for all disjoint Borel sets $(A_n)_{n\in\NN} \subseteq \Rd$ the random variables $(\Lambda(A_n))_{n\in\NN}$ are independent and furthermore satisfy $\Lambda(\cup_{n\in \NN} A_n)=\sum_{n\in \NN} \Lambda(A_n)$. The random measure $\Lambda$ is infinitely divisible if $\Lambda(A)$ is infinitely divisible for all Borel sets $A\subseteq \Rd$.

In this paper we assume that the L\'evy basis $\Lambda$ is stationary and isotropic. With $C(\lambda \dagger Y ) = \log \EE \e^{i\lambda Y}$ denoting the cumulant function for a random variable $Y$, this means that the random variable $\Lambda(A)$ has L\'evy-Khintchine representation
\begin{equation*}
	C(\lambda \dagger \Lambda(A))
	= i \lambda a \abs{A}
	- \frac{1}{2} \lambda^2 \theta \abs{A} + 
	\int_{A \times \R} \bigl( \e^{i\lambda x} - 1 - i\lambda x \II{[-1,1]}(x) \bigr) F(\du, \dd x)
\end{equation*}
for all Borel sets $A\subseteq \Rd$. Here $a\in \R$, $\theta \ge 0$ and $F$ is the product measure $m \otimes \rho$ of Lebesgue measure $m$ and a L\'evy measure $\rho$. We furthermore assume that the L\'evy measure has a regularly varying right tail: Let $\RA$ denote the set of regularly varying functions (at infinity) of index $-\alpha\in\R$, that is, $h\in\RA$ if
\[
	\frac{h(tx)}{h(x)} \to t^{-\alpha}
	\qquad \text{as }
	x\to\infty
\]
for all $t>0$. Note that this class of functions is closed under asymptotic equivalence (at infinity). We say that a distribution $G$ is regularly varying with index $\alpha >0$, writing $G\in\RV$, if its tail $\GG = 1-G\in \RA$ is a regularly varying function of index $-\alpha$. We use similar notation for regularly varying random variables, i.e. random variables with a regularly varying distribution. Concerning the L\'evy measure $\rho$, we thus assume that the right tail is a regularly varying distribution of index $\alpha>0$, which formally reads $\rho({}\cdot \cap (1,\infty))/\rho((1,\infty)) \in\RV$. Note that regularly varying distributions are in particular subexponential, hence $\RV \subseteq \mathcal{S}$; see for instance \cite[Appendix~A.3]{Embrecths1997}.

For the L\'evy basis $\Lambda$ above, we consider the stationary L\'evy-driven field $(X_v)_{v\in\Rd}$ defined~by 
\begin{equation}\label{eq:definitionX}
	X_v = \int_\Rd \ff \Lambda(\dd u) ,
\end{equation}
where $f$ is a positive kernel function.
It can be seen from \cite[Theorem~2.7]{Rajput1989} that the field is well-defined if only there is a $\gamma \in (0,1]$ such that $\int_{\abs{y}>1} \abs{y}^\gamma \rho(\dd y)$ is finite and such that the bounded integration kernel $f:\R^d\to[0,\infty)$ satisfies $\int_\Rd f^\gamma(u) \dd u < \infty$. Throughout the paper we normalize $f$ and furthermore require that it is bounded by a decreasing function only depending on $u$ through $\abs{u}$.

According to \cite[Theorem~2.6]{Potthoff2009} we can (and will) choose a separable version, also in the literature refereed to as a modification, of the random field $(X_v)_{v\in\Rd}$. Similarly we choose a separable version of $(X_v)_{v\in B}$ when we are dealing with a specific index set $B\subseteq \Rd$. Likewise, all other random fields that will be defined throughout the paper will be chosen to be separable versions.

As mentioned in the introduction, we will work under two different assumptions on the L\'evy basis and the kernel function. The assumptions differ in that one is less restrictive with respect to the L\'evy measure, and the other is less restrictive with respect to the integration kernel. We first give the minimal requirements on the pair $(\Lambda,f)$ before presenting the two sets of assumptions. This minimal Assumption~\ref{ass:minimalassumption}, in particular, implies that the field \eqref{eq:definitionX} is well-defined.

\setcounter{assum}{12}
\begin{assum}\label{ass:minimalassumption}
The L\'evy basis $\Lambda$ on $\Rd$ is stationary and isotropic with a L\'evy measure $\rho$ having a regularly varying right tail of index $\alpha > 0$. Moreover, there is $\gamma\in (0,\alpha)$ such that
\begin{equation}\label{eq:gammamomentrho}
	\int_{\abs{y}>1} \abs{y}^\gamma \rho (\dd y)<\infty .
\end{equation}
The integration kernel $f:\Rd\to[0,\infty)$ satisfying $f\le 1$ with $f(0)=1$ is lower semi-continuous and bounded from above by a decreasing c\`adl\`ag function $g$, in the sense that $f(u)\leq g(\abs{u})$ for all $u\in\Rd$. The function $g$ additionally satisfies
\begin{equation}\label{eq:gammamomentkernel}
	\int_\Rd g^{\gamma} (\abs{u}) \dd u < \infty ,
\end{equation}
where $\gamma \in (0,\alpha)\cap (0,1]$ satisfies \eqref{eq:gammamomentrho}.
\end{assum}

Note that the integrability in \eqref{eq:gammamomentrho} along the right tail of $\rho$ is already given from the fact that $\rho$ is regularly varying. In fact, $\int_1^\infty y^\gamma \rho(\dd y)<\infty$ for all $\gamma<\alpha$
(\cite[Proposition~A3.8]{Embrecths1997}). 
Hence, the integrability \eqref{eq:gammamomentrho} is simply a requirement on the existence of moments of the finite measure $\rho({}\cdot \cap (-\infty,-1))$. Furthermore, the integrability of $g^\gamma$ and the fact that $g$ is decreasing implies that
\[
	\int_\Rd \sup_{v\in B} g^{\gamma}(\abs{v-u}) \dd u < \infty
\]
for all fixed sets $B\subseteq \Rd$. Since $f(\cdot)\le g(\abs{\cdot})$ and $\alpha > \gamma$, the claim also holds when substituting $g^{\gamma}(\abs{v-u})$ with $g^{\alpha}(\abs{v-u})$, $f^{\gamma}(v-u)$ or $f^\alpha(v-u)$. Note that the lower semi-continuity of $f$ ensures that the supremum $\sup_{v\in B}$ can be replaced by a suitable countable supremum, making the integrand measurable.

\begin{example}
For any infinitely divisible distribution with a regularly varying right tail there is asymptotic equivalence between the tail of the distribution and the tail of the L\'evy measure; see \cite[Theorem~1]{Embrechts1979}. In particular, the tail of $\Lambda(A)$ will be regularly varying with index $\alpha$ if and only if the L\'evy measure $\rho$ is regularly varying with index $\alpha$.
\end{example}

\begin{example}
Let $0<\alpha<2$. An $\alpha$-stable random field is obtained by letting $\theta=0$ and defining
\[
	\rho(\dd x)=\big(p_-\abs{x}^{-(1+\alpha)}\I{x<0}+p_+x^{-(1+\alpha)}\I{x>0}\big)\dd x,
	\qquad p_+,p_- \ge 0\ , \ p_+ + p_->0,
\]
in the L\'evy-Khintchine representation of $\Lambda$. Here $\rho$ has a regularly varying right (and left) tail with index $\alpha$. 
\end{example}

\begin{example}
If the decreasing upper bound $g$ is a regularly varying function of index $-(d+\epsilon)/\gamma$ for some $\gamma,\epsilon>0$, then \eqref{eq:gammamomentkernel} is satisfied. This includes the simple case where $g(x)= c (1+x)^{-(d+\epsilon)/\gamma}$ is a power function of order $-(d+\epsilon)/\gamma$.
\end{example}

\begin{example}\label{ex:isotropic}
If $h:[0,\infty)\to[0,\infty)$ is a lower semi-continuous function, which satisfies $h\leq 1$, $h(0)=1$ and is bounded from above by a decreasing c\`adl\`ag function $g:[0,\infty)\to[0,\infty)$ satisfying \eqref{eq:gammamomentkernel}, then $f:\Rd\to[0,\infty)$ defined by $f(u)=h(\abs{u})$ will satisfy the requirement from Assumption~\ref{ass:minimalassumption}. With a kernel function of this type, the random field defined by \eqref{eq:definitionX} will be isotropic. In particular, $f$ can be defined by $f(u)=g(\abs{u})$. Note that $g$ is indeed lower semi-continuous, since it is decreasing and c\`adl\`ag.
\end{example}

For the first assumption we also require that $f$ is H\"older continuous.

\begin{assumption}\label{ass:fullassumption1}
The L\'evy basis $\Lambda$ and integration kernel $f$ satisfy Assumption~\ref{ass:minimalassumption}. Moreover, $f$ is  H\"older continuous with some index $\zeta>0$. That is, there is a constant $C$ such that 
\[
	\abs{f(u_1)-f(u_2)} \le C \abs{u_1-u_2}^\zeta
\]
for all $u_1,u_2\in \Rd$.
\end{assumption}

\begin{example}
Let $A$ be a symmetric, positive definite matrix and define $f$ by
\[
f(u)=\exp(-u^TA^{-1}u),
\]
where $u^T$ is a row vector, and $u$ is a column vector. Then $f(0)=1$, $f\leq 1$ and $f$ is H\"older continuous with index 1. Furthermore, $f(\cdot)\leq g(\abs{\cdot})$ with $g(x)=\e^{-x^2/\lambda}$, where $\lambda$ is the largest eigenvalue of $A$. In particular, defining $f(u)=\e^{-\sigma\abs{u}^2}$ yields an isotropic random field, cf. Example~\ref{ex:isotropic}.
\end{example}

The second assumption corresponds to the case where the L\'evy basis is of finite variation, by which we mean that the triple $(a,\theta,\rho)$ of its L\'evy-Khintchine representation satisfies that of a L\'evy process of finite variation.

\begin{assumption}\label{ass:fullassumption2}
The L\'evy basis $\Lambda$ and integration kernel $f$ satisfy Assumption~\ref{ass:minimalassumption}. Moreover, $\Lambda$ has L\'evy-Khintchine representation with $\theta=0$ and a L\'evy measure $\rho$ satisfying
\begin{equation}\label{eq:rhofinitevariation}
\int_{\abs{y}\leq 1}\abs{y}\rho(\dd y)<\infty .
\end{equation}
\end{assumption}

\begin{example}
An $\alpha$-stable random field with $0<\alpha<1$ clearly satisfies the condition in \eqref{eq:rhofinitevariation}.
\end{example}

We are now ready to state the first main result of the paper concerning the tail of $\sup X_v$. 

\begin{theorem}
\label{thm:tailtheorem0}
Let $(X_v)_\vR$ be a L\'evy-driven field given by \eqref{eq:definitionX} where the 
L\'evy basis $\Lambda$ and the kernel function $f$ satisfy either Assumption~\ref{ass:fullassumption1} or Assumption~\ref{ass:fullassumption2}.
Let $B\subseteq \Rd$ be a fixed bounded set. Then
\begin{equation}\label{eq:Xrhoequivalence0}
	\PP \bigl( \sup_{v\in B} X_v > x \bigr)
	\sim
	\rho((x,\infty)) \int_{\Rd} \sup_{v\in B} \ffa \dd u
\end{equation}
as $x\to\infty$.
\end{theorem}

In fact, we show the more general tail result below.
In this theorem, and in the remainder of the paper, we use the notation $y_+=y \I{y\ge 0}$ and $y_+^\gamma = (y_+)^\gamma$ for any $y\in\R$.

\begin{theorem}
\label{thm:tailtheorem}
Let $(X_v)_\vR$ be a L\'evy-driven field given by \eqref{eq:definitionX} where the 
L\'evy basis $\Lambda$ and the kernel function $f$ satisfy either Assumption~\ref{ass:fullassumption1} or Assumption~\ref{ass:fullassumption2}. 
Let $B\subseteq \Rd$ be a fixed bounded set,
and let $(Y_v^1)_v$ be a field independent of $(X_v)_v$ satisfying
\begin{equation}\label{eq:tildeYassumption}
	\EE \bigl(\sup_{v\in B} Y_v^1 \bigr)_+^\beta < \infty
\end{equation}
for some $\beta > \alpha$. Then
\begin{equation*}
	\PP \bigl( \sup_{v\in B} (X_v + Y_v^1) > x \bigr)
	\sim
	\rho((x,\infty)) \int_{\Rd} \sup_{v\in B} \ffa \dd u
\end{equation*}
as $x\to\infty$.
\end{theorem}

Next we turn to the assumption on the expansion of the index sets $(C_n)$ in $\Rd$ relating to the extremal result. For all $n\in\NN$, we require that $C_n$ is a $p$-convex set as defined below. In the following, a compact convex set with non-empty interior is called a \emph{convex body}; see e.g. \cite[Chapter~4]{Schneider1993}.

\begin{definition}
A set $C\subseteq \R^d$ is said to be $p$-convex, if it is connected and has the form
\[
	C=\bigcup_{i=1}^{p} \overline C_i\,,
\]
where $\overline C_1,\ldots,\overline C_{p}$ are convex bodies in $\R^d$. 
\end{definition}

We give the sufficient assumption on the index sets $(C_n)$ used in Theorems~\ref{thm:extremelevy} and \ref{thm:extremelevyv2} below. Due to stationarity of all fields involved,
we can without loss of generality in later results assume that $0\in C_n$ for all $n\in\NN$. Although not formulated in the assumption, this will be assumed in the remainder of the paper. In the assumption, conditions on the so-called \emph{intrinsic volumes} are stated. For $j=0,\dots,d$, the intrinsic volumes $V_j(C)$ describe the geometry of the convex body $C$. For instance, $V_0(C)=1$, $V_1(C)$ is proportional to the mean width, $V_{d-1}(C)$ is half the surface area, and $V_d(C)=\abs{C}$ equals the volume of $C$. Furthermore, the functionals $V_j:\mathcal{K}\to \R$, where $\mathcal{K}$ is the set of all convex bodies, satisfy some important properties, among which we mention
\begin{enumerate}[label=\normalfont(\roman*)] 
	\item They are non-negative, i.e. $V_j(C)\geq 0$ for all $C\in\mathcal{K}$.
	\item They are homogeneous, i.e.  $V_j(\gamma C)=\gamma^j V_j(C)$ for all $\gamma>0$.
	\item They are monotone, i.e. $C\subseteq D$ implies $V_j(C)\leq V_j(D)$.
\end{enumerate}
For a greater exposition of convex bodies and their intrinsic volumes we refer to \cite[Chapter~4]{Schneider1993}.

\begin{assumption}\label{ass:Cnassumption}
The sequence $(C_n)_{n\in\NN}$ consists of $p$-convex bodies, where 
\[
	C_n=\bigcup_{i=1}^{p} C_{n,i}
\]
and $\abs{C_n}\to \infty$ as $n\to\infty$. Furthermore,
\begin{equation}\label{eq:boundedintvolumes}
	\frac{\sum_{i=1}^pV_j(C_{n,i})}{\abs{C_n}^{j/d}}
	\quad
	\text{is bounded in $n$ for each }j=1,\dots,d-1 .
\end{equation}
\end{assumption}

\begin{example}\label{ex:simpleCset}
Let $C=\cup_{i=1}^p \overline C_i$ be a $p$-convex set and define the sequence $(C_n)_{n\in\NN}$ by
\[
	C_n=r_n  C=\bigcup_{i=1}^p r_n \overline C_i ,
\]
where $r_n\uparrow \infty$ as $n\to\infty$. Then 
$\abs{C_n} = \abs{r_n C}=r_n^d \abs{C}$
and $V_j(r_n \overline C_i) = r_n^j V_j(\overline C_i)$ for $j=0,\dots,d$. In particular,
\[
	\frac{\sum_{i=1}^p V_j(r_n \overline C_i)}{\abs{C_n}^{j/d}}
	= \frac{\sum_{i=1}^p V_j(\overline C_i)}{\abs{C}^{j/d}}
\]
is constant in $n$ for all $j=1,\dots,d-1$. Thus the sequence $(C_n)_{n\in\NN}$ satisfies Assumption~\ref{ass:Cnassumption}.
\end{example}

Before presenting the last main results, we recall that the distributions in $\RV$ are exactly the distributions in the maximum domain of attraction $\Fre$ of the Fr\'echet distribution $\Phi_{\alpha}$, where
$
	\Phi_{\alpha}(x) = \exp(-x^{-\alpha})\I{x\ge 0} ;
$
see \cite[Chapter~3]{Embrecths1997}.
In other words, we know that there are norming constants $(\tilde a_n)_n$ with $0<\tilde a_n\to\infty$ such that
\begin{equation*}
	n\rho((\tilde a_n x, \infty)) \to x^{-\alpha}\rho((1,\infty))
\end{equation*}
for all $x>0$. In the remainder of the paper, we let $(a_n)_n$ be a sequence of norming constants of $\rho$ relative to 
$\abs{C_n}$, i.e.
\begin{equation}\label{eq:normingconstants}
	\abs{C_n}\rho((a_n x, \infty)) \to x^{-\alpha}\rho((1,\infty)) 
\end{equation}
for all $x>0$.

\begin{theorem}\label{thm:extremelevy}
Let $(X_v)_\vR$ be a L\'evy-driven field given by \eqref{eq:definitionX} where the 
L\'evy basis $\Lambda$ and the kernel function $f$ satisfy either Assumption~\ref{ass:fullassumption1} or Assumption~\ref{ass:fullassumption2}.
Let $(C_n)_{n\in\NN}$ be a sequence of sets in $\Rd$ satisfying Assumption~\ref{ass:Cnassumption},
and let $a_n$ be the norming constants of the L\'evy measure $\rho$ relative to $\abs{C_n}$, i.e.
$\lim_{n} \abs{C_n} \rho((a_n x , \infty)) = x^{-\alpha} \rho((1,\infty))$ for all $x>0$. 
Then, as $n\to\infty$,
\begin{equation*}
	\PP\Bigl( a_n^{-1}  \sup_{v\in C_n} X_v  \le x \Bigr)
	\to \exp\Bigl( - x^{-\alpha}\rho((1,\infty)) \Bigr)
\end{equation*}
for all $x>0$.
\end{theorem}

The proof techniques used to show Theorem~\ref{thm:extremelevy} almost immediately give us the following more general result. In its formulation, $B(t)=\{u\in\Rd\::\: \abs{u}\le t\}$ denotes the closed ball with radius $t$ and center in the origin $0\in\Rd$. This notation will be used throughout the paper. 

\begin{theorem}\label{thm:extremelevyv2}
Let $(X_v)_\vR$ be a L\'evy-driven field given by \eqref{eq:definitionX} where the 
L\'evy basis $\Lambda$ and the kernel function $f$ satisfy either Assumption~\ref{ass:fullassumption1} or Assumption~\ref{ass:fullassumption2}.
Let $(C_n)_{n\in\NN}$ be a sequence of sets in $\Rd$ satisfying Assumption~\ref{ass:Cnassumption},
and let $a_n$ be the norming constants of the L\'evy measure $\rho$ relative to $\abs{C_n}$, i.e.
$\lim_{n} \abs{C_n} \rho((a_n x , \infty)) = x^{-\alpha} \rho((1,\infty))$ for all $x>0$. Furthermore, let $(Y_v^1)_v$ and $(Y_v^2)_v$ be independent stationary fields independent of $(X_v)_v$, where $(Y_v^1)$ is ergodic and satisfies
\begin{equation}\label{eq:Y1integrability}
	\EE \bigl(\sup_{v\in B(1)} Y_v^1 \bigr)_+^{\beta_1} < \infty
\end{equation}
for some $\beta_1 > \alpha$, and $(Y_v^2)$ satisfies
\begin{equation}\label{eq:Y2integrability}
	\EE  \bigl(\sup_{v\in B(1)} \abs{Y_v^2}\bigr)^{\beta_2} < \infty
\end{equation}
for some $\beta_2>\alpha$.
Then, as $n\to\infty$,
\begin{equation*}
	\PP\Bigl( a_n^{-1}  
	\sup_{v\in C_n} \bigl( X_v + Y_v^1 + Y_v^2 \bigr)  
	\le x \Bigr)
	\to \exp\Bigl( - x^{-\alpha}\rho((1,\infty)) \Bigr)
\end{equation*}
for all $x>0$. 
\end{theorem}

	\section{Results for regularly varying distributions}
	\label{sec:regvar}

In this section we list some useful results concerning regularly varying distributions. 
By assumption,
$\rho$ has a regularly varying right tail, and we thus obtain by Karamata's representation theorem \cite[Theorem~A3.3]{Embrecths1997} that
\begin{equation*}
	\rho((x,\infty)) = 
	c(x) \exp \Bigl( - \int_1^x \frac{\alpha(y)}{y} \dd y \Bigr)
	\qquad\text{for all }
	x\ge 1 ,
\end{equation*}
where $c(x)\to c>0$ and $\alpha(x) \to \alpha$ as $x\to\infty$. Thus, for all positive $\gamma<\alpha$ there exist constants $C>0$ and $x_0\ge 1$ such that
\begin{equation}\label{eq:regvarproductbound}
	\frac{\rho((tx,\infty))}{\rho((x,\infty))}
	\le C t^{-\gamma}
	\qquad \text{for all } t\ge 1 ,\ x\ge x_0 .
\end{equation}

We need bounds similar to \eqref{eq:regvarproductbound} for sums instead of products. The following lemma describes this for arbitrary regularly varying distributions.

\begin{lemma}\label{lem:regvaryingbound}
Let $G\in\RV$ be a regularly varying distribution of index $\alpha > 0$, and let
$\GG$ be its tail. For all $\beta > \alpha$ there are constants $K,C\ge 0$ and $x_0\ge 1$ such that
\begin{equation}\label{eq:regvaryingbound}
	\GG(x-y) \le \GG(x) C \bigl(K + y_+^{\beta}\bigr)
\end{equation}
for all $x\ge x_0$ and $y\in\R$.
\end{lemma}

\begin{proof}
We show that there are constants $C'$ and $K'$ such that
\begin{equation}\label{eq:regvaryingbound0}
	\GG(x-y) \le \GG(x) C' \bigl(K' + y_+\bigr)^{\beta}
\end{equation}
is satisfied for all $x\ge x_0$ and $y\in\R$. From this, \eqref{eq:regvaryingbound} easily follows.

Since $G\in\RV$ and hence $\GG \in \RA$ it follows by Karamata's representation theorem for regularly varying functions that
\begin{equation*}
	\GG(x) = a(x) \exp \Bigl( -\int_1^x \frac{\alpha(z)}{z} \dd z\Bigr)
	\qquad
	\text{for all } x\ge 1 ,
\end{equation*}
for some functions $a(x)\to a>0$ and $\alpha(x)\to \alpha$ as $x\to\infty$.
Fix $\beta >\alpha$ and let $ x_0\ge 1$ be given such that
\begin{equation}\label{eq:regvardistinequality}
	\alpha(x) < \beta ,
	\qquad\text{and}\qquad
	\abs{a(x)-a} \le a/3 
\end{equation}
for all $x\ge  x_0$.

Now consider only $x\ge x_0$ and $y\ge 0$. If $y \le x- x_0$ and hence $x-y \ge  x_0$, we find from 
\eqref{eq:regvardistinequality} above that
\begin{align*}
	\frac{\GG(x-y)}{\GG(x)} &
	= \frac{a(x-y)}{a(x)} \exp\Bigl( \int_{x-y}^x \frac{\alpha(z)}{z} \dd z \Bigr) \\&
	\le 2 \Bigl(\frac{x}{x-y}\Bigr)^{\beta} \\&
	\le 2 \Bigl(1+\frac{y}{ x_0}\Bigr)^\beta.
\end{align*}
If on the other hand $y > x- x_0$, and thus $x-y <  x_0$, we see that
\begin{align*}
	\frac{\GG(x-y)}{\GG(x)} 
	\le \frac{1}{\GG(x)} &
	= \Bigl[ a(x)\exp\Bigl(- \int_{1}^{ x_0} \frac{\alpha(z)}{z} \dd z \Bigr)\Bigr]^{-1} \exp\Bigl( \int_{ x_0}^x \frac{\alpha(z)}{z} \dd z \Bigr) \\&
	\le C_0 \Bigl( \frac{x}{x_0} \Bigr)^{\beta} \\ &
	\le C_0 \Bigl( 1 + \frac{y}{ x_0} \Bigr)^{\beta}
\end{align*}
where the constant $C_0$ can be chosen as
\begin{equation*}
	C_0=\Bigl[ \frac{2a}{3}\exp\Bigl(- \int_{1}^{x_0} \frac{\alpha(z)}{z} \dd z \Bigr)\Bigr]^{-1}.
\end{equation*}
Choosing $K'=x_0$ and $ C'=(K')^{-\beta}\cdot\max\{2, C_0\}$ shows \eqref{eq:regvaryingbound0} for $y\ge 0$.

For $y<0$ the claim \eqref{eq:regvaryingbound0} reads $\GG(x-y) \le \GG(x) C' (K')^{\beta}$, which is clearly true since $x \mapsto \GG(x)$ is decreasing and $C'(K')^{\beta } \ge 2$.
\end{proof}

In the proof of Theorem~\ref{thm:extremelevyv2} we need to find a sequence $d_n$, which tends to infinity at a slower rate than the norming constants $a_n$ of $\rho$, i.e. the sequence satisfying \eqref{eq:normingconstants}.

\begin{lemma}\label{lem:regvarloworder}
Let $\beta > \alpha$ where $\alpha$ is the regularly varying index of $\rho$. There is a sequence $d_n \to \infty$ of order $d_n=o(a_n)$ such that
\[
	\lim_{n\to \infty} 
	\frac{\overline H(d_n)}{\rho((a_n,\infty))} = 0
\]
for all distributions $H$ on $[0,\infty)$ satisfying $\int x^\beta H(\dd x) < \infty$.
\end{lemma}

\begin{proof}
Since $\rho$ is regularly varying of index $\alpha>0$, there is a slowly varying function $L$ such that $\rho((x,\infty))=x^{-\alpha}L(x)$. Let $\epsilon = \beta - \alpha>0$ and define $d_n = a_n^{1-\epsilon/(2\beta)}$. Then $d_n \to \infty$ and $d_n=o(a_n)$ as $n\to\infty$. Noticing that $a_n^{\epsilon/2} L(a_n)\to \infty$ and $d_n^\beta \overline H(d_n) \to 0$ concludes the proof.
\end{proof}

	\section{Boundedness and tail behavior}
	\label{sec:tail}

In this and the following sections we will need a decomposition of $(X_v)$ into independent fields. For this we write $\Lambda=\Lambda_1+\Lambda_2$ as the independent sum of two L\'evy bases with L\'evy-Khintchine representations
\begin{align*}
	C(\lambda \dagger \Lambda_1(A))
	&=	\int_{A \times \R} \bigl( \e^{i\lambda x} - 1\bigr) F_{(1,\infty)}(\du, \dd x)\\
		C(\lambda \dagger \Lambda_2(A))
	&= i \lambda a \abs{A}
	- \frac{1}{2} \lambda^2 \theta \abs{A} + 
	\int_{A \times \R} \bigl( \e^{i\lambda x} - 1 - i\lambda x \II{[-1,1]}(x) \bigr) F_{(-\infty,1]}(\du, \dd x)\,,
\end{align*}
respectively. 
Here, $F_D=m\otimes \rho_D$, where $\rho_D$ is the restriction of $\rho$ to $D\subseteq \R$. Similarly, we decompose the field $(X_v)_v$ into a sum of two independent random fields $X_v = Z_v + Y_v$
for all $\vR$, where
\begin{align}
	\label{eq:Zdefinition}
	& 
	Z_v = \int_\Rd \ff \Lambda_1(\dd u) \qquad\text{and}\\ &
	\label{eq:Ydefinition}
	Y_v = \int_\Rd \ff \Lambda_2(\dd u).
\end{align}
In this section we will make use of a further decomposition of $\Lambda_2$ into $\Lambda_2=\Lambda_{2,C}+\Lambda_{2,-}$, where
\begin{align*}
	C(\lambda \dagger \Lambda_{2,-}(A))
	&=	\int_{A \times \R} \bigl( \e^{i\lambda x} - 1\bigr) F_{(-\infty,-1)}(\du, \dd x),\\
		C(\lambda \dagger \Lambda_{2,C}(A))
	&= i \lambda a \abs{A}  - \frac{1}{2} \lambda^2 \theta \abs{A}+
	\int_{A \times \R} \bigl( \e^{i\lambda x} - 1 - i\lambda x \II{[-1,1]}(x) \bigr) F_{[-1,1]}(\du, \dd x).
\end{align*}
We let $Y_v=Y_v^C+Y_v^-$ be the corresponding independent decomposition of $(Y_v)_v$.

It will be essential in the subsequent proofs that the sample paths of $(X_v)_{v\in B}$ are bounded, when $B$ is a fixed bounded set. For this we will study the different parts of the decomposition separately.

\begin{lemma}\label{lem:fieldcontinuity}
Under Assumption~\ref{ass:fullassumption1} the field $(Y_v^C)_{v\in B}$ has a continuous version. In particular, it holds that $\PP(\sup_{v\in B}\abs{Y_v^C}<\infty)=1$. 
\end{lemma}

\begin{proof}
Let $\Lambda_{2,C}'$ be the so-called spot variable of $\Lambda_{2,C}$, i.e. a random variable equivalent in distribution to $\Lambda_{2,C}(S)$ for sets $S\subseteq \Rd$ with Lebesgue measure $1$. Then all moments of $\Lambda_{2,C}'$ are finite, and consequently all moments of $Y^C_v$ are finite. Following the lines of the proof of \cite[Theorem~5.1]{Stehr2020b} with a Lipschitz assumption replaced by the H\"older assumption in Assumption~\ref{ass:fullassumption1}, it is immediately seen that there exist finite constants $C'_n$, chosen independently of $r$ and $v\in B$, and natural numbers $n' \ge n/2$, such that
\[
	\EE [( Y_{v+r}-Y_v)^n ] \le C'_n \abs{r}^{\zeta {n'}}
\]
with the equality $n'=n/2$ whenever $n$ is even. Using the fact that $\abs{r}\le\mathrm{diam}(B)$, we find finite $C'\ge 0$ 
and $\eta > n>4d/\zeta$ such that
\begin{equation*}
	\EE \abs*{Y^C_{v+r}-Y^C_v}^{n}  
	\le C'_{n} \abs{r}^{2d}\abs{r}^{\zeta n/2 - 2d}
	\le \frac{C' \abs{r}^{2d}}{ \abs[\big]{\log\abs{r}}^{1+\eta}}
\end{equation*}
for all $v\in B$. From a corollary to \cite[Theorem~3.2.5]{Adler1981} we conclude that $(Y_v)_{v\in B}$ has a continuous version on $B$.
\end{proof}

It remains to show that Assumption~\ref{ass:fullassumption1} implies that $(Z_v)_{v\in B}$ and $(Y^-_v)_{v\in B}$ are both bounded. For this, the representation in Theorem~\ref{thm:compoundpoissonrepresentation} below will be useful. This theorem is an adapted version of \cite[Theorem~3.4.3]{Samorodnitsky2016} and in the formulation we apply a similar (when possible) notation. 
We say that $(\Gamma_n)$ is a sequence of standard Poisson arrivals, if
\[
\Gamma_n=e_1+\dots + e_n,
\]
where $(e_n)$ are independent and all follow an exponential distribution with mean 1.

The theorem is formulated in terms of $(X_v)_{v\in B}$ as defined under the minimal Assumption~\ref{ass:minimalassumption}.
It will thus be applicable to $(Z_v)_{v\in B}$ and $(Y_v^-)_{v\in B}$ defined under Assumptions~\ref{ass:fullassumption1} and \ref{ass:fullassumption2}.

\begin{theorem}\label{thm:compoundpoissonrepresentation}
Assume, in addition to Assumption~\ref{ass:minimalassumption}, that the support of the L\'evy measure $\rho$ is a subset of $(0,\infty)$, that $\int_0^1 x \rho(\dd x)<\infty$, and that $\Lambda(A)$ has L\'evy-Khintchine representation
\[
C(\lambda \dagger \Lambda(A))=\int_{A \times \R} \bigl( \e^{i\lambda x} - 1\bigr) F(\du, \dd x) 
\] 
with $F=m\otimes\rho$. Let $\mu$ be a probability measure on $\R^d$ that is equivalent with Lebesgue measure $m$. Let
\[
r(v)=\frac{\dd m}{\dd\mu}(v)
\]
for $v\in\R^d$ and define
\[
\rho_r(v,\cdot)=r(v)\rho(\cdot)
\] 
for each $v\in \R^d$.  Furthermore, define the generalized inverse of $x\mapsto \rho_r(v,(x,\infty))$ as
\[
G(x,v)=\inf\{y>0\::\: \rho_r(v,(y,\infty))\leq x\}.
\]
Let $(U_n)$ be a sequence of independent and identically distributed random variables with common distribution $\mu$ and let $(\Gamma_n)$ be a sequence of standard Poisson arrivals on $(0,\infty)$ independent of $(U_n)$. Then $(\tilde{X}_v)_{v\in B}$ defined by
\begin{equation}\label{eq:ztildedef}
\tilde{X}_v=\sum_{n=1}^\infty G(\Gamma_n,U_n)f(v-U_n)
\end{equation}
for all $v\in B$ is a well-defined random field, which is equal in its finite-dimensional distributions to $(X_v)_{v\in B}$
\end{theorem}

\begin{proof}
The proof follows the lines of the proof of \cite[Theorem~3.4.3]{Samorodnitsky2016} closely, utilizing the fact that both $G$ and $f$ are non-negative and, with the notation $[\![\cdot]\!]$ from \cite{Samorodnitsky2016}, that both 
\[
\int_{\Rd}\int_\R [\![f(v-u)x]\!]\,\rho(\dd x) \dd u \quad\text{and}\quad \int_{\Rd}\int_\R [\![f(v-u)x]\!]^2\,\rho(\dd x) \dd u
\]
are bounded from above by
\[
\int_{\Rd}\int_{0<x\leq 1} f(v-u) x \,\rho(\dd x) \dd u+2\int_{\Rd}\int_{x>1} f(v-u)^\gamma x^\gamma \,\rho(\dd x) \dd u ,
\]
which is finite by assumption.
\end{proof}

\begin{corollary}\label{cor:ztilde}
The field $(\tilde{X}_v)_{v\in B}$ defined in \eqref{eq:ztildedef} is bounded on a set with probability 1. If furthermore, the kernel function $f$ is continuous, then   the field is continuous with probability~1.
\end{corollary}

\begin{proof}
The $n$th term in \eqref{eq:ztildedef} is non-negative and bounded (uniformly in $v\in B$) from above by $G(\Gamma_n,U_n)\sup_{u\in B}f(u-U_n)$. Since the series 
\begin{equation}\label{eq:uniformupperbound}
\sum_{n=1}^\infty G(\Gamma_n,U_n)\sup_{u\in B}f(u-U_n)
\end{equation}
converges almost surely by \cite[Theorem~3.4.1]{Samorodnitsky2016}, the boundedness follows. If, additionally, $f$ is continuous then each term in \eqref{eq:ztildedef} is continuous, and the convergence of the series is uniform due to \eqref{eq:uniformupperbound}. Thereby the continuity of $(\tilde{X}_v)_{v\in B}$ follows.
\end{proof}

The corollary almost immediately gives the desired result for $(Z_v)_{v\in B}$ and $(Y^-_v)_{v\in B}$.

\begin{lemma}\label{lem:compoundcont}
Under Assumption~\ref{ass:fullassumption1} the fields $(Z_v)$ and $(Y^-_v)$ both have continuous versions. In particular, they are almost surely bounded, i.e. 
\[
\PP(\sup_{v\in B}\abs{Z_v}<\infty)=\PP(\sup_{v\in B}\abs{Y_v^-}<\infty)=1.
\]
\end{lemma}

\begin{proof}
We give the argument for $(Z_v)$. The proof for $(Y_v^-)$ is, except for a change in notation in Theorem~\ref{thm:compoundpoissonrepresentation} and Corollary~\ref{cor:ztilde}, identical. First note that $(Z_v)_{v\in B}$ satisfies the conditions of Theorem~\ref{thm:compoundpoissonrepresentation}. Recall that we have chosen a separable version of $(Z_v)_{v\in B}$ and let $T\subset B$ be the corresponding countable separating dense subset. By Corollary~\ref{cor:ztilde} and the equality of distributions given in Theorem~\ref{thm:compoundpoissonrepresentation} the field $(Z_v)_{v\in B}$ is uniformly continuous on $T$ with probability 1.  In particular, it has a continuous version on $B$. 
\end{proof}

\begin{lemma}\label{lem:finitevarbounded}
Under Assumption~\ref{ass:fullassumption2} the field $(X_v)_{v\in B}$ has a version that is almost surely bounded.
\end{lemma}

\begin{proof}
We can decompose $\Lambda=\Lambda^++\Lambda^-+\tilde{a}m$, where $\tilde{a}m$ is the scaled Lebesgue measure for some constant $\tilde{a}$, and $\Lambda^+$ and $\Lambda^-$ are independent, given such that
\begin{align*}
	C(\lambda \dagger \Lambda^+(A))
	&=	\int_{A \times \R} \bigl( \e^{i\lambda x} - 1\bigr) F_{(0,\infty)}(\du, \dd x)\\
	C(\lambda \dagger \Lambda^-(A))
	&=	\int_{A \times \R} \bigl( \e^{i\lambda x} - 1\bigr) F_{(-\infty,0)}(\du, \dd x).
\end{align*}
respectively. We let $(X_v^+)$ and $(X_v^-)$ be the random fields obtained by integrating with respect to $\Lambda^+$ and $\Lambda^-$. It suffices to show that each of $(X_v^+)_{v\in B}$ and $(X_v^-)_{v\in B}$ have versions with bounded sample paths. We focus on $(X_v^+)_{v\in B}$, since the proof for $(X_v^-)_{v\in B}$ is identical. Note that $(X_v^+)_{v\in B}$ satisfies the conditions of Theorem~\ref{thm:compoundpoissonrepresentation} such that, by Corollary~\ref{cor:ztilde}, there exists $(\tilde{X}_v^+)_{v\in B}$ that is almost surely bounded and has the same finite dimensional distributions as $(X_v^+)_{v\in B }$. Recall that we have chosen a separable version of $(X_v^+)_{v\in B}$ and let $T\subset B$ be the corresponding countable separating dense subset. Now, the distributions of $(X_v^+)_{v\in T}$ and $(\tilde{X}_v^+)_{v\in T}$ are identical, which concludes the proof.
\end{proof}

\begin{proof}[{\bf Proof of Theorem~\ref{thm:tailtheorem0}}]
Fix the bounded index set $B\subseteq \Rd$.
Let $m$ denote Lebesgue measure, and define the function $H$ by
\[
	H(x) = m \otimes \rho 
	\bigl( \{(u,z)\in\Rd\times \R \::\: \sup_{v\in B} z \ff > x\} \bigr)
\]
for $x> 0$.
For convenience, define $s_u = (\sup_{v\in B} \ff)^{-1}\ge 1$ for all $u\in\Rd$. With this notation at hand, we see that
\begin{align*}
	\frac{H(x)}{\rho((x,\infty))} 
	= \int_\Rd \frac{\rho((s_u x,\infty))}{\rho((x,\infty))} \dd u ,
\end{align*}
and since $\rho$ has a regularly varying right tail of index $\alpha$, we obtain the convergence
\[
	\frac{\rho((s_ux,\infty))}{\rho((x,\infty))}
	\to
	s_u^{-\alpha},
	\qquad \text{as } x\to\infty ,
\]
for all $u\in\Rd$. Now choose $\gamma<\alpha$ according to \eqref{eq:gammamomentkernel} satisfying that $g^\gamma$ is integrable. 
This implies in particular that $f^\gamma$ is integrable, and also $u\mapsto s_u^{-\gamma}$ is so. Since $s_u\ge 1$ for all $u\in\Rd$ we find by dominated convergence, using the bound \eqref{eq:regvarproductbound}, that
\[
	\frac{H(x)}{\rho((x,\infty))} \to \int_\Rd \bigl(\sup_{v\in B }\ff\bigr)^{\alpha} \dd u
	= \int_\Rd \sup_{v\in B } \ffa \dd u < \infty
\]
as $x\to\infty$. Therefore, the distribution $1-\min\{H,1\}\in \RV$, and thus it is especially in the class of subexponential distributions. 
Recall that $(X_v)_{v\in B}$ is chosen separable and let $T\subset B$ be the corresponding countable separating dense subset. By Lemmas~\ref{lem:fieldcontinuity}--\ref{lem:finitevarbounded}, the field $(X_v)$ is almost surely bounded on $B$,
\[
	\PP\bigl( \sup_{v\in B} \abs{X_v} < \infty\bigr)=	\PP\bigl( \sup_{v\in T} \abs{X_v} < \infty\bigr) = 1 ,
\]
and, using that $\sup_{v\in B} X_v$ and $\sup_{v\in T} X_v$ coincide, we conclude by \cite[Theorem~3.1]{RosinskiSamorodnitsky1993} that
\begin{equation*}
	\PP\bigl( \sup_{v\in B} X_v > x \bigr)
	\sim
	H(x)
	\sim
	\rho((x,\infty)) \int_\Rd \sup_{v\in B } \ffa \dd u
\end{equation*}
as $x\to\infty$.
\end{proof}

\begin{proof}[{\bf Proof of Theorem~\ref{thm:tailtheorem}}]
Fix the bounded index set $B\subseteq \Rd$.
Let $F$ denote the distribution of $\sup_{v\in B} X_v$, which, by Theorem~\ref{thm:tailtheorem0} and the fact that $\rho$ is regularly varying, satisfies $F\in \RV$. In particular, $F$ is subexponential and hence
\begin{equation}\label{eq:Xrhoequivalence1}
	\frac{\FF(x - y)}
	{\FF(x)}
	\to 1
\end{equation}
for all $y\in\R$ as $x\to\infty$. Let $\pi^1$ denote the distribution of $(Y_v^1)_v$, and let $y^*=\sup_{v\in B} y_v$ for a deterministic field $(y_v)_v$. Then
\begin{align*}
	\frac{\PP(\sup_{v\in B} (X_v + Y_v^1)>x)}{\FF(x)} &
	\le \frac{\PP(\sup_{v\in B} X_v > x - \sup_{v\in B} Y_v^1)}
	{\FF(x)} \\ &
	= \int \frac{\FF(x-y^*)}{\FF(x)} \pi^1 (\dd y) ,
\end{align*} 
with the integrand tending to $1$ by \eqref{eq:Xrhoequivalence1} above. Now fix $\beta>\alpha$ such that \eqref{eq:tildeYassumption} is satisfied. By Lemma~\ref{lem:regvaryingbound}, the integrand has an $x$-independent and integrable upper bound, and we thus obtain by dominated convergence and \eqref{eq:Xrhoequivalence0} that
\[
	\limsup_{x\to\infty} 
	\frac{\PP(\sup_{v\in B} (X_v + Y_v^1)>x)}{\rho((x,\infty))}
	\le \int_\Rd \sup_{v\in B} \ffa \dd u .
\]
Since $(\inf_{v\in B} Y_v^1)_+$ also has $\beta$-moment, we similarly find that
\[
	\liminf_{x\to\infty} 
	\frac{\PP(\sup_{v\in B} (X_v + Y_v^1)>x)}{\rho((x,\infty))}
	\ge \int_\Rd \sup_{v\in B} \ffa \dd u ,
\]
proving the claim.
\end{proof}

	\section{The geometry}
	\label{sec:geometry}

The purpose of this section is to introduce the geometry necessary to the proof of the extremal results Theorems~\ref{thm:extremelevy} and \ref{thm:extremelevyv2} above. In particular, we construct certain boxes which will be used to approximate the continuous index sets $(C_n)$, and we show their limiting behavior. Recall by Assumption~\ref{ass:Cnassumption} that the index sets $(C_n)$ are $p$-convex bodies with relatively bounded intrinsic volumes as in \eqref{eq:boundedintvolumes}.

We first introduce some new notation which will be used throughout the remainder of the paper. For two sets $A, B\subseteq \R^d$, we define their Minkowski sum by $A\oplus B=\{a+b\mid a\in A,b\in B\}$. When referring to (discrete or continuous) cubes with side-length equal to $s> 0$, we mean a box where all side-lengths equal $s$. Moreover, we say that $C_r(u)$ is the closed $r$-cube with corner $u\in\Rd$ if
\[
	C_r(u) = u + [0,r]^d
\]
for $r>0$. Lastly, we let $\mathrm{dist}(u,A)$ be the distance from $u\in\Rd$ to the set $A\subseteq \Rd$.

The following corollary, which is shown in \cite{StehrRonnNielsen2020}, is a consequence of the famous Steiner Theorem from convex geometry stating that the volume of the Minkowski sum of a convex body in $\Rd$ and a ball $B(r)$ of radius $r$ is a polynomial of degree $d$ in $r$.

\begin{corollary}[{\cite[Corollary~1]{StehrRonnNielsen2020}}]\label{cor:intrinsic}
Let $C\subseteq\R^d$ be a convex body with boundary $\partial C$ and let $r>0$. Then
\[
\sum_{j=0}^{d-1} \omega_{d-j}V_{j}(C)r^{d-j}\leq \abs{\partial C\oplus B(r)}\leq 2\sum_{j=0}^{d-1} \omega_{d-j}V_{j}(C)r^{d-j} ,
\] 
where $V_j(C)$ is the $j$th intrinsic volume of $C$ and $\omega_j$ is the volume of the $j$-dimensional unit ball.
\end{corollary}

For the sequence of index sets $(C_n)_{n\in\NN}$ satisfying Assumption~\ref{ass:Cnassumption} we will for each $k,L\in\NN$ define $\Tt\in\NN$ by
\[
	\Tt = L\cdot \floor*{\sqrt[d]{\abs{C_n}/k}}.
\] 
The integer $\Tt$ will serve as the side-length of a sequence of cubes approximating the index sets $C_n$. Formally, for each $z=(z_1,\ldots,z_d)\in\ZZ^d$ and $n$ large enough relative to $k$, we define the cube $\Iz\subseteq\Rd$ as
\begin{align*}
		\Iz & 
		= z \, \Tt + [0,\Tt)^d \\ &
		= \bigtimes_{i=1}^d [z_i\, \Tt, (z_i + 1)\Tt ) .
\end{align*}
Furthermore, we let $\Pp$ be the set of indices $z\in\ZZ^d$ for which $\Iz$ is contained in $C_n$, and we let $\Qq$ be the set of indices $z$ for which $\Iz$ intersects $C_n$:
\[
	\Pp = \{ z\in\Zd \::\: \Iz \subseteq C_n\}
	\quad\text{and}\quad 
	\Qq = \{ z\in\Zd \::\: \Iz \cap C_n \neq \emptyset\} . 
\]
We let the number of such inner and outer approximating boxes be denoted by $\pp=\abs{\Pp}$ and $\qq=\abs{\Qq}$, respectively.
Note that, by construction, $\pp\le k/L^d$ and $\qq\ge k/L^d$ for values of $n$ large enough relative to $k$. 
When proving our results we essentially divide $\Iz$ into cubes of side-length $L$, $C_L(v)$ for $v\in(L\ZZ)^d$. To this end, define the grid points of $\Iz$ by $\Jz=\Iz \cap (L\ZZ)^d$, and note that $\abs{\Jz} = (\Tt)^d/L^d \sim \abs{C_n}/k$ as $n\to\infty$.
Lastly, we define a set of grid points which, when continuously filled with cubes $C_L(v)$, approximate $C_n$ from the inside and outside, respectively:
\[
	\dnm = \bigcup_{z\in\Pp} \Jz 
	\quad\text{and}\quad
	\dnp = \bigcup_{z\in\Qq} \Jz ,
\]
which then satisfy
\begin{equation}\label{eq:DnBnInequality}
	\bigcup_{z\in\dnm} C_L(z) 
	\subseteq 
	C_n 
	\subseteq
	\bigcup_{z\in\dnp} C_L(z)  .
\end{equation}

\begin{theorem}\label{thm:maingeometrictheorem}
Let $(C_n)_{n\in\NN}$ satisfy Assumption~\ref{ass:Cnassumption}. Then,
\begin{enumerate}[label=\normalfont(\roman*)]
	\item \label{eq:geomthm2} for all $L\in\NN$ the sequences $\pp$ and $\qq$, defined above, satisfy that
	\[
		\liminf_{n\to\infty} \pp \sim \frac{k}{L^d}
		\qquad\text{and}\qquad
		\limsup_{n\to\infty} \qq \sim \frac{k}{L^d}
	\]
	as $k\to\infty$, 
	\item\label{eq:geomthm3} for each $k$, $L$ and $n$ with $n$ large enough relative to $k$, it holds that $\dnp\subseteq K_{n,k,L}$, where $K_{n,k,L}$ is the cube
	\[
		K_{n,k,L}=\bigcup_{z\in \Nk } \Jz,    
	\]
	and $\Nk$ is on the form $\Nk=[-c_{k,L} \:,\: c_{k,L}]^d\cap \ZZ^d$ for some $c_{k,L}<\infty$,
	\item\label{eq:geomthm4} for all $L\in\NN$ there exists $c_L<\infty$ such that for all $n\in\NN$
	\[
		\dnp \subseteq \KL=\big[-c_L  \cdot \abs{C_n}^{1/d} \:,\: c_L\cdot \abs{C_n}^{1/d}\big]^d\cap (L\ZZ)^d .
	\]
\end{enumerate}
\end{theorem}

We note that the discrete index set $\Nk$ in \ref{eq:geomthm3} above does not depend on $n$. In particular, for any fixed $k$ and $L$, all $\dnp$ are contained in the same finite collection of (increasing) sets of grid points $\Jz$.

\begin{proof}
For each $n,k,L\in\NN$ define $\tildet=L\,\bigl(\abs{C_n}/ k \bigr)^{1/d}$. Turning to Corollary~\ref{cor:intrinsic}, we obtain
\[
	\abs{\partial C_n\oplus B(\tildet)}
	\le \sum_{i=1}^p \abs{\partial C_{n,i}\oplus B(\tildet)}
	\le 2\sum_{j=0}^{d-1} \omega_{d-j}
	\Bigl(\sum_{i=1}^p V_{j}(C_{n,i})\Bigr)\tildet^{d-j} ,
\]
which clearly implies that
\[
	\frac{L^d}{k}\frac{1}{\tildet^d}\abs{\partial C_n\oplus B(\tildet)}
	\le 2\sum_{j=0}^{d-1} \omega_{d-j}
	\frac{\sum_{i=1}^pV_{j}(C_{n,i})}{\abs{C_n}^{j/d}}
	\Bigl(\frac{L}{\sqrt[d]{k}}\Bigr)^{d-j} .
\]
By \eqref{eq:boundedintvolumes} of Assumption~\ref{ass:Cnassumption} we conclude the convergence
\[
	\limsup_{n\to\infty}
	\frac{L^d}{k}\frac{1}{\tildet^d}\abs{\partial C_n\oplus B(\tildet)}
	\to 0
\] 
as $k\to\infty$. 
In fact, replacing $L$ by 
$r\cdot L$ for any $r\in\NN$ 
and realizing that $\tilde t_{n,k,r\cdot L}=r\,\tildet$ show that also
\[
	\limsup_{n\to\infty}
	\frac{L^d}{k}\frac{1}{\tildet^d}\abs{\partial C_n\oplus B(r\, \tildet)}
	\to 0
\] 
as $k\to\infty$ for all $r\in\NN$.  
Since $\Tt\le \tildet$, and $\Tt \sim \tildet$ as $n\to\infty$, we also have
\begin{equation}\label{fml:tnklimit}
	\limsup_{n\to\infty}
	\frac{L^d}{k}\frac{1}{\Tt^d}	
	\abs{\partial C_n\oplus B(r \, \Tt)} \to 0
\end{equation}
as $k\to\infty$ for all $r\in\NN$. Since $\Tt$ denotes the side-length of the cubes $\Iz$, we can find an integer $r\in\NN$ independent of $n,k,L$ such that 
$\Iz \subseteq \partial C_n \oplus B(r\, \Tt)$
whenever $\Iz\cap \partial C_n \neq \emptyset$. Using that $\abs{\Iz}=\Tt^d$, we then find
\begin{align*}
	\frac{L^d}{k}(\qq-\pp) &
	\le 
	\frac{L^d}{k}\frac{1}{\Tt^d}
	\abs[\Big]{
	\bigcup_{\Iz\cap \partial C_n\neq \emptyset}
	\Iz } \\ &
	\le \frac{L^d}{k}\frac{1}{\Tt^d}
	\abs{ \partial C_n\oplus B(r\, \Tt)}.
\end{align*}
Together with \eqref{fml:tnklimit} and the fact that $\pp \le k/L^d \le \qq$ for $n$ large enough relative to $k$, this proves statement~\ref{eq:geomthm2}.

We now show statement~\ref{eq:geomthm3}. For each fixed $k$ and $L$, the sequence $(\qq)_n$ is bounded by a constant $c_{k,L}$ for $n$ large enough relative to $k$. Furthermore, as the origin $0\in C_n$, it must also satisfy $0\in \Qq$. By assumption, $C_n$ is connected, and thus $\Qq$ consists of at most $c_{k,L}$ points that are pairwise neighbors. In particular,
\[
	\Qq
	\subseteq 
	[-c_{k,L} \:,\: c_{k,L}]^d \cap \Zd  ,
\]
which implies the desired result. Finally, \ref{eq:geomthm4} follows from \ref{eq:geomthm3}. 
\end{proof}

Before moving on to proving our remaining main results, we recall a lemma from \cite{StehrRonnNielsen2020} which will be useful throughout the paper:

\begin{lemma}[{\cite[Lemma~1]{StehrRonnNielsen2020}}]\label{lem:Steinerintegral}
Let $C$ be a convex body and let $r\ge 0$ be fixed. For any decreasing function $g:[0,\infty) \to [0,\infty)$ satisfying $\int_0^\infty g(x) x^{d-1} \dd x<\infty$, it holds that
\begin{align}
	\label{eq:Steinerintegral1}
	\int_{(C \oplus B(r))^c} \sup_{v\in C} g(\abs{v-u}) \dd u
	 & = 
	\sum_{j=0}^{d-1} \mu_j V_j(C) \int_{r}^\infty g(x) x^{d-j-1} \dd x 
	\qquad \text{and}\\
	\label{eq:Steinerintegral2}
	\int_{C} \sup_{v\in (C \oplus B(r))^c} g(\abs{v-u}) \dd u
	 & \le 
	\sum_{j=0}^{d-1} \mu_j V_j(C) \int_{r}^\infty g(x) x^{d-j-1} \dd x
\end{align}
for constants $\mu_j$, $j=0,\dots,d-1$, independent of $C$.
\end{lemma}

	\section{Extremal results}
	\label{sec:extreme}

In this section we prove the main extreme value results Theorems~\ref{thm:extremelevy} and \ref{thm:extremelevyv2} 
assuming either Assumption~\ref{ass:fullassumption1} or Assumption~\ref{ass:fullassumption2}. Throughout this section we work under either of the assumptions, which in particular implies that the minimal Assumption~\ref{ass:minimalassumption} is satisfied.

In proving the extreme value results,
we define a field $(\zvt)$ with distribution which, when $t\to\infty$, approximates the distribution of $(Z_v)$ defined in \eqref{eq:Zdefinition}. For fixed $t>0$ we let $f^{(t)}(\cdot) = f(\cdot) \I{\abs{\cdot} < t}$ and define the field
\begin{equation}\label{eq:Ztdefinition}
	\zvt = \int_\Rd \fft \Lambda_1(\dd u) 
	= \int_{\{\abs{v-u} < t\}} \ff \Lambda_1(\dd u).
\end{equation}
Note that $\zvt$ and $Z_{v'}^{(t)}$ are independent for $\abs{v-v'}>2t$, and furthermore $\zvt \le Z_v$ for any $t$. We will explicitly use this to bound the asymptotic distribution function of $\sup_{C_n} Z_v$ with that of $\sup_{C_n} \zvt$, which is more easily managed due to independence. In proving our main results we will only need $t$ to take a fixed finite value and the value $\infty$. Many of the lemmas below will be formulated in terms of any finite or infinite choice of $t$, with $t=\infty$ corresponding to the field $Z_v^{(\infty)}=Z_v$. As it turns out, the final result will not depend on the fixed finite $t$.

\begin{lemma}\label{lem:Lintegralequivalence}
For all $0<t\le \infty$,
\begin{equation*}
	\int_\Rd \sup_{v\in C_L(0)} \bigl(\fft\bigr)^\alpha \dd u
	\sim L^d 
\end{equation*}
as $L\to\infty$.
\end{lemma}

\begin{proof}
Since $f\le 1$ with $f(0)=1$ and $\abs{C_L(0)}=L^d$, it is seen that the claim follows once we show that
\[
	\int_{(C_L(0))^c} \sup_{v\in C_L(0)} \bigl(\fft \bigr)^{\alpha} \dd u = o(L^{d})
\]
as $L\to\infty$. Since $f^{(t)}(\cdot)$ is increasing in $t$ and bounded by $g(\abs{\cdot})$, this is satisfied for all $t$ if
\begin{equation}\label{eq:L_g_integral}
	\int_{(C_L(0))^c} \sup_{v\in C_L(0)} \gga \dd u =  o(L^{d})
\end{equation}
as $L\to\infty$. We now show this claim.

Since the cube $C_L(0)\subseteq \Rd$ is a convex body and $g^{\alpha}$ is decreasing, we find by Lemma~\ref{lem:Steinerintegral}, equation \eqref{eq:Steinerintegral1}, and the homogeneity of the intrinsic volumes that there are constants $\mu_j$ independent of $C_L(0)$ such that
\begin{align}
	\nonumber
	\int_{(C_L(0))^c} \sup_{v\in C_L(0)} \gga \dd u &
	= \sum_{j=0}^{d-1} \mu_j V_j(C_L(0)) \int_0^\infty g^{\alpha}(x) x^{d-j-1} \dd x 
	\\ &
	\label{eq:gintegralSteiner}
	= \sum_{j=0}^{d-1} \mu_j V_j(C_1(0)) L^j \int_0^\infty g^{\alpha}(x) x^{d-j-1} \dd x .	
\end{align}
By assumption, $g^{\alpha}(\abs{\cdot})$ is integrable, or equivalently,
\[
	\int_0^\infty g^\alpha(x) x^{d-1} \dd x < \infty ,
\]
and we conclude that the integral in \eqref{eq:gintegralSteiner} is finite for all $j=0,\dots,d-1$. Thus, dividing \eqref{eq:gintegralSteiner} by $L^d$ and letting $L\to\infty$ show that \eqref{eq:L_g_integral} is satisfied. This concludes the proof.
\end{proof}

The lemma below follows by arguments similar to that of Theorem~\ref{thm:tailtheorem0} realizing that $\zvt$ is almost surely bounded on bounded sets. 

\begin{lemma}\label{lem:Zrhoequivalence}
Let $B\subseteq \Rd$ be a fixed bounded set. Then, for all $0<t\le \infty$,
\begin{align*}
	\PP \bigl( \sup_{v\in B} \zvt > x \bigr)
	\sim
	\rho((x,\infty)) \int_{\Rd} \sup_{v\in B} \bigl(\fft \bigr)^{\alpha} \dd u
\end{align*}
as $x\to\infty$. In particular, $\sup_{v\in B} \zvt \in \RV$.
\end{lemma}

Since $\rho$ is regularly varying and in particular in the maximum domain of attraction of the Fr\'echet distribution (i.e. it satisfies \eqref{eq:normingconstants}), Lemma~\ref{lem:Zrhoequivalence} in particular implies that
\begin{equation}\label{eq:supZtconvergence}
	\abs{C_n} \PP \bigl( \sup_{v\in B} \zvt > a_n x \bigr)
	\to x^{-\alpha} \rho((1,\infty)) 
	\int_{\Rd} \sup_{v\in B} \bigl(\fft \bigr)^{\alpha} \dd u
\end{equation}
for any fixed set $B\subseteq \Rd$ and any $t>0$.

To ease notation in the remainder of the paper, we define $\taut$ as
\begin{equation*}
	\taut = 
	x^{-\alpha} \rho((1,\infty)) \int_{\Rd} \sup_{v\in C_L(0)} \bigl(\fft \bigr)^{\alpha} \dd u
\end{equation*}
for any $L\in\NN$, $t>0$ and fixed $x>0$. For $t=\infty$ we write $\tau_L=\tau_L^{(\infty)}$. Note that, by Lemma~\ref{lem:Lintegralequivalence}, $\taut$ satisfies
\begin{equation}\label{eq:tauLconvergence}
	\frac{\taut}{L^d} \to x^{-\alpha} \rho((1,\infty))
\end{equation}
as $L\to\infty$.

The lemma below will be used repeatedly in this section. It ensures in particular that the field $(Y_v)_v$ has moments of all orders and thus is of minor importance when determining the extremal behavior of the L\'evy-driven field $(X_v)_v$.

\begin{lemma}\label{lem:expintegrability}
Let the field $(Y_v)_\vR$ be given by \eqref{eq:Ydefinition}. Then
\begin{equation}\label{eq:supYexponential}
	\EE \exp\bigl(\epsilon \sup_{v\in C_L(0)} Y_v \bigr)<\infty
\end{equation} 
for all $\epsilon>0$ and all $L\in\NN$.
\end{lemma}

\begin{proof}
Write $Y_v=Y_v^C + Y_v^-$ as the independent decomposition given in Section~\ref{sec:tail}. Utilizing the fact that $(Y_v^-)$ is a non-positive field, it is clear that \eqref{eq:supYexponential} follows once we show that
\begin{equation}\label{eq:supYCexponential}
	\EE \exp\bigl(\epsilon \sup_{v\in C_L(0)} Y_v^C \bigr)<\infty .
\end{equation} 
for all $\epsilon>0$ and all $L\in\NN$.
Let $T\subset C_L(0)$ be the countable separating dense subset associated to the separable field $(Y_v^C)_{v\in C_L(0)}$.
By considerations as in \cite{RNielsen2016} and \cite{Stehr2020b}, the countable field $(Y_v^C)_{v\in T}$ is infinitely divisible with a characteristic function given as in \cite[Eq.~(1.1)]{Bravermann1995}. Its L\'evy measure $\nu$ defined on $\R^{T}$ is given as follows: Let $H:\Rd \times [-1,1] \to \R^{T}$ be given as
\[
	H(u,x) = (x \ff)_{v\in T} ,
\]
and let $\nu = (m \otimes \rho_{[-1,1]}) \circ H^{-1}$ be the image-measure of $m \otimes \rho_{[-1,1]}$. Here $m$ denotes Lebesgue measure and $\rho_{[-1,1]}$ is the L\'evy measure of the basis $\Lambda_{2,C}$ defining $Y_v^C$.
By Lemmas~\ref{lem:fieldcontinuity} and \ref{lem:finitevarbounded} we find that
\[
	\PP(\sup_{v\in C_L(0)} \abs{Y_v^C} < \infty) 
	= \PP(\sup_{v\in T} \abs{Y_v^C} < \infty)
	=1 ,
\]
and, by construction, $\nu( \{z \in \R^{T}\::\: \sup_{T} \abs{z_v} > 1\}) = 0$.
By an application of \cite[Lemma~2.1]{Bravermann1995} we conclude that $\EE \exp(\epsilon \sup_{v\in C_L(0)} \abs{Y_v^C})< \infty$ for all $\epsilon$, which clearly implies \eqref{eq:supYCexponential}.
\end{proof}

The following theorem, which shows that the field $(Y_v)_\vR$ is ergodic, is exactly Corollary~4 of \cite{StehrRonnNielsen2020}. For a greater exposition of spatial ergodicity, we refer to their Section~5. In the theorem and in the remainder of the paper, we write $J_z$ for $\Jz$ to avoid too heavy notation. Furthermore, we recall the set $\Nk$ initially defined in Theorem~\ref{thm:maingeometrictheorem}\ref{eq:geomthm3}.

\begin{theorem}[{\cite[Corollary~5]{StehrRonnNielsen2020}}]\label{thm:mixingY}
Let the field $(Y_v)_\vR$ be given by \eqref{eq:Ydefinition}, and let $h$ be a function satisfying $\EE \abs{h((Y_u)_{u\in C_L(0)})}<\infty$.
For all $z\in \Nk$ it then holds that
\[
	\frac{1}{\abs{J_z}}\sum_{v\in J_z} h((Y_{u+v})_{u\in C_L(0)})
	\to \EE h((Y_u)_{u\in C_L(0)})
\]
almost surely as $n\to\infty$. The result also holds true if $J_z$ is replaced with a subset of $J_z$ in the shape of a box, which increases in size asymptotically as $J_z$.
\end{theorem}

From now on we let $x>0$ be fixed, and we let the sequence $(a_n)$ be the norming constants of $\rho$ satisfying \eqref{eq:normingconstants}.

\begin{lemma}\label{thm:partialy}
Let $(\zvt)_v$ and $(Y_v)_v$ be given by \eqref{eq:Ztdefinition} and \eqref{eq:Ydefinition}, respectively. Then, for almost all realizations $(y_v)_v$ of $(Y_v)_v$, it holds for all $z\in \Nk$ and for all $0<t\le \infty$ that
\begin{equation}\label{eq:partialy1}
	\frac{\abs{C_n}}{\abs{J_z}} \sum_{v\in J_z}
	\PP\bigl( \sup_{u\in C_L(v)} (Z_u^{(t)} + y_u) > a_nx \bigr) \to 
	\taut 
\end{equation}
as $n\to\infty$.
The result also holds true if $J_z$ is replaced with a subset of $J_z$ in the shape of a box, which increases in size asymptotically as $J_z$.
\end{lemma}

\begin{proof}
We only show the convergence \eqref{eq:partialy1} as the expression for $J_z$ replaced by an asymptotically size-equivalent box follows identically. 

Throughout this proof we fix $\beta > \alpha$ such that
\begin{equation}\label{eq:partialproofYgamma}
	\EE \bigl(\sup_{v\in C_L(0)} Y_v \bigr)_+^\beta < \infty ,
\end{equation}
which is possible due to Lemma~\ref{lem:expintegrability}.

Let $F_{L,t}$ denote the distribution of $\sup_{u\in C_L(v)} Z_u^{(t)}$ for $t>0$, and note that $F_{L,t}\in\RV$ by Lemma~\ref{lem:Zrhoequivalence}. By \eqref{eq:supZtconvergence},
\begin{equation*}\label{eq:Ztconvergencexn0}
	\abs{C_n} \FF_{L,t}(a_nx) 
	\to \taut
\end{equation*}
as $n\to\infty$, where $\FF_{L,t}$ is the tail of $F_{L,t}$. Since in particular $F_{L,t}\in \mathcal{S}$, the convergence
\begin{equation*}
	\frac{\FF_{L,t}(a_nx - y)}{\FF_{L,t}(a_nx)} \to 1
\end{equation*}
as $n\to\infty$ is uniform for all $\abs{y}\le N$, for all $N\in\NN$; see e.g. \cite[Definition~2.1]{Pakes2004}. Hence, with $\supy = \sup_{C_L(v)} y_u$, we find by an application of Lemma~\ref{lem:regvaryingbound} and Theorem~\ref{thm:mixingY}, utilizing \eqref{eq:partialproofYgamma}, that
\begin{align*}
	\MoveEqLeft
	\frac{\abs{C_n}}{\abs{J_z}} 
	\sum_{v\in J_z} \FF_{L,t}(a_nx - \supy)\\ &
	\le \frac{\abs{C_n} \FF_{L,t}(a_nx)}{\abs{J_z}}
	\sum_{v\in J_z} \Bigl(
	\frac{\FF_{L,t}(a_nx - \supy)}{\FF_{L,t}(a_nx)} \I{\abs{\supy}\le N} \\ &
	\hspace{35mm}
	+ C \bigl(K+(\supy)_+^\beta \bigr) \I{\abs{\supy}> N} \Bigr) \\ &
	\to \taut \EE \Bigl(\I{\abs{\sup_{C_L(0)} Y_v} \le N}
	+ C \bigl(K+(\sup_{C_L(0)} Y_v)_+^\beta \bigr) 
	\I{\abs{\sup_{C_L(0)} Y_v} > N} \Bigr)
\end{align*}
as $n\to\infty$, almost surely. Letting $N\to\infty$ shows by monotone convergence that
\[
	\limsup_{n\to\infty} \frac{\abs{C_n}}{\abs{J_z}} 
	\sum_{v\in J_z} \FF_{L,t}(a_nx - \supy) \le \taut
\]
almost surely. By similar arguments we find that also
\[
	\liminf_{n\to\infty} \frac{\abs{C_n}}{\abs{J_z}} 
	\sum_{v\in J_z} \FF_{L,t}(a_nx - \bary) \ge \taut ,
\]
where we used the notation $\bary = \inf_{C_L(v)} y_u$. As
\[
	\FF_{L,t}(a_nx - \bary)
	\le 
	\PP\bigl( \sup_{u\in C_L(v)} (Z_u^{(t)} + y_u) > a_nx \bigr)
	\le 
	\FF_{L,t}(a_nx - \supy) ,
\]
the convergence \eqref{eq:partialy1} follows.
\end{proof}

We say that two subsets $A,B$ of $\Zd$ or $\Rd$ are $r$-separated if $\mathrm{dist}(A,B) = \inf\{\abs{a-b}\::\:a\in A,b\in B\}\ge r$ and there are two disjoint sets $A',B'\subseteq \Rd$, both of which are connected, such that $A \subseteq A'$ and $B \subseteq B'$. In particular in the coming lemmas, we consider subsets of $(L\ZZ)^d$ and require that these are $\gamma_n$-separated, where, even though we do not explicitly state it, we assume that $\gamma_n\ge L$. 

Before proceeding, we define the last piece of convenient notation. For a set $A\subseteq (L\ZZ)^d$ and a deterministic field $(y_v)_v$ we let
\[
	\myt(A) = \max_{z\in A} \sup_{u\in C_L(z)} (Z^{(t)}_u + y_u) ,
\]
where we simply write $\myL$ for $\myL^{(\infty)}$, that is, when considering the field $(Z_v)$.

\begin{lemma}\label{lem:mixingZ}
Let $(Z_v)_v$ and $(Y_v)_v$ be given by \eqref{eq:Zdefinition} and \eqref{eq:Ydefinition}, respectively. There is a sequence $\gamma_n=o(\sqrt[d]{\abs{C_n}})$ such that for all $\gamma_n$-separated
sets $A,B\subseteq \KL\subseteq (L\ZZ)^d$, where at least one is a box, it holds that
\begin{equation*}
	\abs[\big]{
	\PP (\myL(A \cup B) \le a_nx ) -
	\PP (\myL(A) \le a_nx )\PP (\myL(B) \le a_nx )
	}
	\le \alphaL,
\end{equation*}
where $\alphaL\to 0$ as $n\to\infty$ for all $L\in\NN$ and almost all realizations $(y_v)_v$ of $(Y_v)_v$. 
\end{lemma}

\begin{proof}
We choose $\gamma_n/2=\abs{C_n}^{1/d-\delta}$ for some $0<\delta<\min\{\frac{1}{d\alpha},\frac{1}{2d}\}$. In particular, $\gamma_n= o(\sqrt[d]{\abs{C_n}})$ as $n\to\infty$.
Let the sets $A$ and $B$ be given as in the lemma and define
\begin{equation*}
	\Aa = \bigcup_{v\in A}C_L(v)
	\qquad\text{and}\qquad
	\Bb = \bigcup_{v\in B}C_L(v)
	\qquad\text{and}\qquad
	\Kk = \bigcup_{v\in \KL} C_L(v) .
\end{equation*}
Throughout the proof, we assume that $A$ is a box, and thus $\Aa$ is a continuous box and in particular a convex body.
Recall that $B(r)$ denotes the closed ball in $\Rd$ of radius $r\ge 0$ with center $0\in\Rd$. Define $\Aa_{n}=\Aa\oplus \ball{\gamma_n/2}$ and note that $\Bb\subseteq \Bb_n=\bigl( \Aa\oplus \ball{\gamma_n} \bigr)^c=\bigl( \Aa_n\oplus \ball{\gamma_n/2} \bigr)^c$, since $\Aa$ and $\Bb$ are $\gamma_n$-separated (strictly speaking they are only $(\gamma_n-L)$-separated, however, this is equivalent as $n\to\infty$). 

For all $v\in\Aa$ let
\begin{equation*}
	Z_v^A = \int_{\Aa_{n}} \ff \Lambda_1(\dd u),
	\qquad\text{and}\qquad
	\overline Z_v^A =\int_{\Aa_{n}^c} \ff \Lambda_1(\dd u) .
\end{equation*}
Similarly, for all $v\in\Bb$,
\begin{equation*}
	Z_v^B = \int_{\Aa_{n}^c} \ff \Lambda_1(\dd u),
	\qquad\text{and}\qquad
	\overline Z_v^B =\int_{\Aa_{n}} \ff \Lambda_1(\dd u) .
\end{equation*}
Since $\Lambda_1$ is a positive measure all $Z_v^A$, $\overline{Z}_v^A$, $Z_v^B$ and $\overline{Z}_v^B$ are non-negative, and we have
\begin{equation*}
	\sup_{v\in\Aa}\overline{Z}^A_v
	\le \int_{\Aa_{n}^c}\sup_{v\in \Aa} \ff \Lambda_1(\dd u)
\end{equation*}
and
\begin{equation*}
	\sup_{v\in\Bb}\overline{Z}^B_v
	\le \int_{\Aa_{n}}\sup_{v\in \Bb}\ff \Lambda_1(\dd u)\,.
\end{equation*}
If $\alpha< 1$, we choose $\gamma<\alpha<\beta$ such that $\frac{\gamma}{\beta}>\frac{d-1}{d}$ and such that the upper bound $g$ has $\gamma$-moment (in accordance with Assumption~\ref{ass:minimalassumption}). If $\alpha\geq 1$, we let $\gamma=1$ and choose $\beta>\alpha$ such that $\frac{1}{\beta}>d\delta$ (recall that $\frac{1}{\alpha}>d\delta$). Let $\epsilon_n=\abs{C_n}^{1/\beta}$, which clearly satisfies $\epsilon_n\to\infty$, and define the events $S_n^A$ and $S_n^B$ by
\[
	S^A_n= 
	\bigl( \sup_{v\in\Aa}\overline{Z}^A_v \le \epsilon_n \bigr)
	\qquad\text{and}\qquad 
	S^B_n= 
	\bigl( \sup_{v\in\Bb}\overline{Z}^B_v \le \epsilon_n \bigr) .
\]
First, we show that the probability $\PP((S^A_n)^c)$ has an $A$-independent upper bound tending to $0$ as $n\to\infty$. For a fixed $n$, we define the sequence $(r_m)_{m\geq 1}$ by $r_1=\gamma_n/2$ and recursively for $m>1$ by
\[
	\abs[\big]{(\Aa\oplus B(r_{m+1}))\setminus (\Aa\oplus B(r_m))}=1.
\]
For each $m\ge 1$ define the sets
\[
	E_m=(\Aa\oplus B(r_{m+1}))\setminus (\Aa\oplus B(r_m))
\]
of unit volume,
and note that their union equals $\Aa_n^c$ by construction. Utilizing that $\Lambda_1$ is a positive measure, that $g$ is decreasing and that $g(\abs{\cdot})$ constitutes an upper bound to $f(\cdot)$, we find
\begin{equation*}
\int_{\Aa_{n}^c}\sup_{v\in \Aa} \ff \Lambda_1(\dd u) 	\le \sum_{m=1}^\infty g(r_m)\Lambda_1(E_m).
\end{equation*}
Let $\Lambda_1'$ be the spot variable of the L\'evy basis $\Lambda_1$, that is, a random variable equivalent in distribution to $\Lambda_1(S)$ with $S\subseteq \Rd$ satisfying $\abs{S}=1$. Using Markov's inequality we now find that
\begin{equation}\label{eq:SAsetinequality}
	\PP((S^A_n)^c)  
	\le \frac{1}{\epsilon_n^\gamma}\EE\Bigl(\sum_{m=1}^\infty g(r_m)^\gamma \Lambda_1(E_m)^\gamma\Bigr)	
	= \frac{1}{\epsilon_n^\gamma}\EE((\Lambda_1')^\gamma)\sum_{m=1}^\infty g(r_m)^\gamma .
\end{equation}
Note that the underlying L\'evy measure $\rho_1$ has finite $\gamma$-moment (see e.g. \cite[Proposition~A3.8]{Embrecths1997}), and thus also $\Lambda_1'$ has finite $\gamma$-moment (cf. \cite[Theorem~25.3]{Sato1999}). Showing that $\PP((S^A_n)^c)\to 0$ uniformly in $A$ therefore amounts to showing that the sum above is of order $o(\epsilon_n^\gamma)$ uniformly in $A$. Clearly, it is enough to consider the sum starting at index $m=2$, which is more convenient. Using the construction of the sets $E_m$ and the fact that $g$ is decreasing, Lemma~\ref{lem:Steinerintegral} and the construction of $\Kk$ implies the existence of $n$- and $A$-independent constants $\mu_j$ and $\tilde{\mu}_j$ such that 
\begin{equation}\label{eq:ggammaintegrability}
\begin{aligned}
	\sum_{m=2}^\infty g(r_m)^\gamma  &
	= \sum_{m=2}^\infty \int_{E_{m-1}} g(r_m)^\gamma \dd u\\ &
	\le \sum_{m=1}^\infty \int_{E_{m}} \sup_{v\in \Aa} g(\abs{v-u})^\gamma \dd u\\ &
	= \int_{\Aa_n^c} \sup_{v\in \Aa} g(\abs{v-u})^\gamma \dd u \\ &
	= \sum_{j=0}^{d-1} \mu_j V_j(\Aa) \int_{\gamma_n/2}^\infty g(x)^\gamma x^{d-j-1} \dd x \\ &
	\le \sum_{j=0}^{d-1} \mu_j V_j(\Kk) \int_{\gamma_n/2}^\infty g(x)^\gamma x^{d-j-1} \dd x \\ &	
	\le \sum_{j=0}^{d-1}\tilde{\mu}_j \abs{C_n}^{j/d}\int_{\gamma_n/2}^\infty g(x)^\gamma x^{d-j-1} \dd x .
\end{aligned}
\end{equation}
For each $j=0,\dots,d-1$ we have
\begin{equation}\label{eq:twointegrals}
\begin{aligned}
	\MoveEqLeft
	\abs{C_n}^{j/d} \int_{\gamma_n/2}^{\infty} g(x)^\gamma x^{d-j-1}\,\dd x \\ &
	\le 
	\abs{C_n}^{(d-1)/d}\int_{\gamma_n/2}^{\infty} g(x)^\gamma \,\dd x
	+\int_{\gamma_n/2}^{\infty} g(x)^\gamma x^{d-1}\,\dd x.
\end{aligned}
\end{equation}
The second integral on the right hand side clearly tends to $0$, and in particular it is of order $o(\epsilon_n^\gamma)$ as $n\to\infty$. In the case $\alpha<1$, the integrability of $g^\gamma$ and the construction of $\epsilon_n$, gives that the first term on the right hand side is of order $o(\epsilon_n^\gamma)$ as $n\to\infty$. In the case $\alpha\ge 1$, where $\gamma=1$, we have for the first term that
\begin{align*}
	\frac{1}{\epsilon_n}\abs{C_n}^{(d-1)/d}\int_{\gamma_n/2}^{\infty} g(x) \,\dd x &
	\le \abs{C_n}^{d\delta-1/\beta}\Bigl(\frac{\gamma_n}{2}\Bigr)^{d-1}\int_{\gamma_n/2}^{\infty} g(x) \,\dd x\\ &
	\le \abs{C_n}^{d\delta-1/\beta}\int_{\gamma_n/2}^{\infty} g(x)x^{d-1} \,\dd x
\end{align*}
which tends to 0, since we have chosen $\beta,\delta$ with $1/\beta>d\delta$. 
Thus, combining \eqref{eq:SAsetinequality}--\eqref{eq:twointegrals}, we have shown for all $\alpha>0$ that $\PP((S_n^A)^c)$ has an upper bound independent of $A$ that tends to 0 as $n\to\infty$. 

The procedure for evaluating $\PP((S_n^B)^c)$ is similar: Let $n$ be fixed and define for each $r>0$ the set
\[
	\Aa_n^{-r} = \Aa_n \setminus (\partial \Aa_n \oplus B(r)) ,
\]
recalling that $\partial \Aa_n$ denotes the boundary of $\Aa_n$. Hence, $\Aa_n^{-r}$ is the set $\Aa_n$ without the strip of thickness $r$ closest to its boundary.
Note that $\Aa_n^{-r}$ is decreasing in $r$ and that it will be empty for $r$ large enough. Define for all $m\ge 0$ the sequence $s_m=r_m - \gamma_n/2$ with $(r_m)$ as above, and define for $m\ge 1$ the sets
\[
	F_m= \Aa_n^{-s_{m}} \setminus \Aa_n^{-s_{m+1}}.
\]
As a consequence of the proof of Lemma~\ref{lem:Steinerintegral}, equation \eqref{eq:Steinerintegral2} (see \cite{StehrRonnNielsen2020}), we obtain that
\begin{align*}
	\abs{F_m} &
	= \abs{\Aa_n^{-s_m}} - \abs{\Aa_n^{-s_{m+1}}} \\ &
	\le \abs{\Aa_n \oplus B(s_{m+1})} - \abs{\Aa_n \oplus B(s_{m})} \\ &
	= \abs{\Aa \oplus B(r_{m+1})} - \abs{\Aa \oplus B(r_{m})} \\ &
	= \abs{E_m} ,
\end{align*}
with $E_m$ defined as above.
Since $\Lambda_1$ is non-negative, we e.g. have that $\EE(\Lambda_1(F_m)^\gamma)\le \EE(\Lambda_1(E_m)^\gamma)$, and then, similarly to \eqref{eq:SAsetinequality}, we find
\[
	\PP((S^B_n)^c)  
	\le \frac{1}{\epsilon_n^\gamma} \EE((\Lambda_1')^\gamma)\sum_{m=1}^\infty g(r_m)^\gamma.
\]
We can proceed as above to find an upper bound that tends to 0 and is independent of $B$.

Identically to arguments applied in the proof of \cite[Lemma~6]{StehrRonnNielsen2020} we can derive that
\begin{equation}\label{eq:mixinginequality}
\begin{aligned}
&	\abs[\big]{
	\PP (\myL(A \cup B) \le a_nx ) -
	\PP (\myL(A) \le a_nx )\PP (\myL(B) \le a_nx )
	}\\
&\le 2\Bigl(\sum_{v\in \KL}
		\PP(M_{L,y}(\{v\}) > a_nx-\epsilon_n)-
	\sum_{v\in \KL}
		\PP(M_{L,y}(\{v\}) > a_nx+\epsilon_n)\Bigr)\\
		&\phantom{=}+2 \Bigl( \PP \bigl((S^A_n)^c \bigr)
	+  \PP \bigl( (S^B_n)^c\bigr) \Bigr) .
\end{aligned}
\end{equation}
We have already seen that the second term above tends to 0 as $n\to\infty$. 
Since $a_n=\abs{C_n}^{1/\alpha}S(\abs{C_n})$ for a slowly varying function $S$ (cf. \cite[p. 131]{Embrecths1997}), we have $\epsilon_n=o(a_n)$. 
From this it is not difficult to see that $\rho((a_n x, \infty)) \sim \rho((a_n x \pm \epsilon_n, \infty))$ as $n\to\infty$.
Therefore, realizing that $\abs{C_n}\PP(M_{L,y}(\{v\})>a_nx + \epsilon_n)$ and $\abs{C_n}\PP(M_{L,y}(\{v\})>a_nx-\epsilon_n)$ have the same limit for fixed $v$, the considerations that led to \eqref{eq:partialy1} also show that the two sums in the first term of \eqref{eq:mixinginequality} have the same limit as $n\to\infty$.
This completes the proof.
\end{proof}

In the remainder of the paper, we consider the sequence $\gamma_n$ given in Lemma~\ref{lem:mixingZ}. The following generalization follows easily by induction.

\begin{lemma}\label{lem:mixinggeneral}
Let $(Z_v)_v$ and $(Y_v)_v$ be given by \eqref{eq:Zdefinition} and \eqref{eq:Ydefinition}, respectively, and let $(y_v)_v$ be a realization of $(Y_v)_v$. Let for $r\in\NN$ the boxes $A_1,\dots, A_r\subseteq \KL$ be pairwise $\gamma_n$-separated boxes. Then
\begin{equation*}
	\abs[\Big]{
	\PP \Bigl(\bigcap_{i=1}^r \{\myL(A_i) \le a_nx \} \Bigr) -
	\prod_{i=1}^r \PP (\myL(A_i) \le a_nx )	}
	\le (r-1) \alphaL .
\end{equation*}
\end{lemma}

By construction we have $\gamma_n<\Tt$ for $n$ sufficiently large relative to fixed $k$ and $L$. For each $z\in \Nk$ we divide $J_z=\Jz \subseteq (L\ZZ)^d$ into two disjoint subsets, $H_z=H_z^{n,k,L}$ and $H_z^*=H_z^{*,n,k,L}$, where  
\begin{equation*}
	H_z ^{n,k,L}  
	= \bigl\{
		u \in J_z \::\:
		z_j \Tt \le u_j < (z_j+1) \Tt - \gamma_n,\:
		\text{for all } j=1,\dots,d
	\bigr\}
\end{equation*}
is a box constructed such that it is $\gamma_n$-separated from $H_{u}^{n,k,L}$ for all $v\neq u\in \Nk$.
From now on we simply write $H_z$ and $H_z^*=J_z \setminus H_z$ to ease notation, but bare in mind the dependence on $n,k,L$.
In the proof of Lemma~\ref{lem:supBnbound} we explicitly use the fact that $H_z^*$ equals the union of a set of overlapping boxes: Defining $B_{z,1}^*,\dots,B_{z,d}^*$ by
\begin{equation*}
	B_{z,j}^* = 
	\bigl\{
		u \in J_z \::\:
		(z_j +1)\Tt - \gamma_n\le u_j < (z_j+1) \Tt
	\bigr\} 
\end{equation*}
for all $j=1,\dots,d$, it is easily seen that their union equals $H_z^*$. Note that they are of size $\abs{B_{z,j}^*} \sim \Tt^{d-1} \gamma_n/L^d$ as $n\to\infty$.

\begin{lemma}\label{lem:maxinequalitypqlevy}
Let $(Z_v)_v$ and $(Y_v)_v$ be given by \eqref{eq:Zdefinition} and \eqref{eq:Ydefinition}, respectively, and let $(y_v)_v$ be a realization of $(Y_v)_v$. Then it holds that
\begin{equation}\label{eq:maxinequality1plevy}
\begin{aligned}
	\MoveEqLeft	
	\abs[\Big]{
	\PP ( \myL(\dnm) \le a_nx ) -
	\prod_{z\in \Pp} \PP ( \myL(J_z) \le a_nx )	} \\&
	\le  
	2 \sum_{z\in \Pp}\PP ( \myL(H_z) \le a_nx < \myL(H_z^*) ) + (\pp-1) \alphaL ,
\end{aligned}
\end{equation}
and similarly
\begin{equation}\label{eq:maxinequality1qlevy}
\begin{aligned}
	\MoveEqLeft
	\abs[\Big]{
	\PP ( \myL(\dnp) \le a_nx ) -
	\prod_{z\in \Qq} \PP ( \myL(J_z) \le a_nx )	} \\&
	\le  
	2 \sum_{z\in \Qq} \PP ( \myL(H_z) \le a_nx < \myL(H_z^*) ) + (\qq-1) \alphaL.
\end{aligned}
\end{equation}
\end{lemma}

\begin{proof}
As $H_z$ is a subset of $J_z$, we find that
\begin{align*}
	0 & \le 
	\PP\Bigl( \bigcap_{z\in \Pp}\{\myL(H_z)\le a_nx\} \Bigr)-\PP\bigl(\myL(\dnm)\le a_nx\bigr) \\&
	\le \PP\Bigl( \bigcup_{z\in \Pp} \{\myL(H_z)\le a_nx <  \myL(H_z^*)\} \Bigr) \\&
	\le \sum_{z\in \Pp}\PP \bigl( \myL(H_z) \le a_nx < \myL(H_z^*) \bigr).
\end{align*}
By construction of $H_z$, an application of Lemma~\ref{lem:mixinggeneral} shows that 
\begin{equation*}
	\abs[\Big]{
	\PP \Bigl(\bigcap_{z\in \Pp} \{\myL(H_z) \le a_nx \} \Bigr) -
	\prod_{z\in \Pp} \PP \bigl( \myL(H_z) \le a_nx \bigr)	}
	\le (\pp-1) \alphaL ,
\end{equation*}
which, combined with the fact that
\begin{align*}
	0 & 
	\le \prod_{z\in \Pp} \PP \bigl( \myL(H_z) \le a_nx \bigr) - 
	\prod_{z\in \Pp} \PP \bigl( \myL(J_z) \le a_nx \bigr) \\&
	\le 
	\sum_{z\in \Pp} \Bigl(
	\PP \bigl( \myL(H_z) \le a_nx \bigr)
	-
	\PP \bigl( \myL(J_z) \le a_nx \bigr)
	\Bigr)\\&
	= 
	 \sum_{z\in \Pp}\PP \bigl( \myL(H_z) \le a_nx < \myL(H_z^*) \bigr),
\end{align*}
concludes \eqref{eq:maxinequality1plevy}. The claim \eqref{eq:maxinequality1qlevy} follows similarly.
\end{proof}

\begin{lemma}\label{lem:supBnbound}
Let $(Z_v)_v$ and $(Y_v)_v$ be given by \eqref{eq:Zdefinition} and \eqref{eq:Ydefinition}, respectively. Then, for almost all realizations $(y_v)_v$ of $(Y_v)_v$, the following is satisfied
\begin{equation}\label{eq:Rboundcont}
	\begin{aligned}
    \Big(\liminf_{n\to\infty} \min_{z\in \Nk}\PP( \myL(J_z) \le a_nx)\Big)^{\tilde q_{k,L}}&
	\le
	\liminf_{n\to\infty} \PP \bigl( \sup_{v\in C_n} (Z_v + y_v) \le a_nx \bigr) \\&
	\le 
	\limsup_{n\to\infty} \PP \bigl( \sup_{v\in C_n} (Z_v + y_v) \le a_nx \bigr)\\&
	\le \Big(\limsup_{n\to\infty}\max_{z\in \Nk} \PP( \myL(J_z) \le a_nx)\Big)^{\tilde p_{k,L}} ,
	\end{aligned}
\end{equation}
where $\tilde p_{k,L} = \liminf_{n}\pp$ and $\tilde q_{k,L} = \limsup_{n}\qq$.
\end{lemma}

\begin{proof}
Let $R_{n,k,L}^{p}\ge 0$ and $R_{n,k,L}^{q}\ge 0$ denote the upper bounds in \eqref{eq:maxinequality1plevy} and \eqref{eq:maxinequality1qlevy}, respectively.
For all $z\in \Nk$, we have
\begin{align*}
	\PP ( \myL(H_z) \le a_nx < \myL(H_z^*)) &
	\le \PP(\myL(H_z^*) > a_nx) \\&
	\le \sum_{j=1}^d \PP(\myL(B_{z,j}^*) > a_nx) \\ &
	\le \sum_{j=1}^d \sum_{v\in B_{z,j}^*}\PP \bigl( \sup_{u\in C_L(v)}(Z_u + y_u) > a_nx \bigr),
\end{align*}
where we recall that $\abs{B_{z,j}^*}\sim\Tt^{d-1} \gamma_n/L^d$ as $n\to\infty$. 
In particular, $\abs{B_{z,j}^*} = o(\abs{J_z})$ as $n\to\infty$ for all $j=1,\dots,d$, and hence $J_z\setminus B_{z,j}^*$ is a box, which increases in size asymptotically as $J_z$.
Since the limit appearing in Lemma~\ref{thm:partialy} is finite and $\abs{J_z}, \abs{J_z\setminus B_{z,j}^*}$ and $\abs{C_n}$ are asymptotically of the same order, we conclude by Lemma~\ref{thm:partialy} that
\begin{align*}
	\PP ( \myL(H_z) \le a_nx < \myL(H_z^*)) &
	\le  
	\sum_{j=1}^d 
		\sum_{v\in J_z}
		\PP \bigl( \sup_{u\in C_L(v)}(Z_u + y_u) > a_nx \bigr) 
		\\ & \qquad
		- \sum_{j=1}^d  	
		\sum_{v\in J_z\setminus B_{z,j}^*}
		\PP \bigl( \sup_{u\in C_L(v)}(Z_u + y_u) > a_nx \bigr)
			\\ &
	\to 0 
\end{align*}
almost surely for all $z\in \Nk$.
By Lemma~\ref{lem:mixingZ} we thus obtain that
\begin{equation*}
	\lim_{n\to\infty} R_{n,k,L}^p 
	= \lim_{n\to\infty} R_{n,k,L}^q
	= 0 
\end{equation*}
almost surely.
Turning to \eqref{eq:DnBnInequality} and using Lemma~\ref{lem:maxinequalitypqlevy} show that
\begin{equation*}
	\begin{aligned}
	\liminf_{n\to\infty}\prod_{z\in \Qq} \PP( \myL(J_z) \le a_nx)&
	\le
	\liminf_{n\to\infty} \PP \bigl( \sup_{v\in C_n} (Z_v + y_v) \le a_nx \bigr) \\&
	\le 
	\limsup_{n\to\infty} \PP \bigl( \sup_{v\in C_n} (Z_v + y_v) \le a_nx \bigr)\\&
	\le \limsup_{n\to\infty} \prod_{z\in \Pp} \PP( \myL(J_z) \le a_nx) .
	\end{aligned}
\end{equation*}
As all factors in the products are probabilities and thus lie in the interval $[0,1]$, it is easily seen that \eqref{eq:Rboundcont} is satisfied.
\end{proof}

In the following lemma, the summation over $\{v<v'\in J_z\}$ indicates the double sum of points in $v\in J_z$ and subsequent points $v'\in J_z$ falling strictly after $v$ under some underlying enumeration. This notation will be used in the remainder of the paper.

\begin{lemma}\label{lem:anticlustering}
Let $(\zvt)_v$ and $(Y_v)_v$ be given by \eqref{eq:Ztdefinition} and \eqref{eq:Ydefinition}, respectively. For all $0<t<\infty$ there is a sequence of functions $g_L$ satisfying 
\begin{equation}\label{eq:anticlusteringfunction}
		\limsup_{k\to\infty} k \, g_L(k) = o(L^d)
\end{equation}
as $L\to\infty$ such that
\begin{equation}\label{eq:anticlusteringZy}
	\limsup_{n\to\infty} \sum_{v<v' \in J_z} 
	\PP \bigl( 
	\sup_{u\in C_L(v)} (Z_u^{(t)} + y_u)> a_nx,
	\sup_{u\in C_L(v')} (Z_u^{(t)} + y_u)> a_nx
	\bigr)
	\le g_L(k)
\end{equation}
for all $z\in\Nk$ and almost all realizations $(y_v)_v$ of $(Y_v)_v$.
\end{lemma}

\begin{proof}
Throughout this proof we fix $\beta > \alpha$ such that
\[
	\EE \bigl(\sup_{v\in C_L(0)} Y_v \bigr)_+^\beta < \infty ,
\]
which is possible due to Lemma~\ref{lem:expintegrability}.
Furthermore,
let $F_{L,t}\in\RV$ denote the regularly varying distribution of $\sup_{u\in C_L(v)} Z_u^{(t)}$ for $t>0$, and recall that 
\begin{equation}\label{eq:Ztconvergencexn}
	\abs{C_n} \FF_{L,t}(a_nx) 
	\to \taut
\end{equation}
as $n\to\infty$.

Fix $L> 2t$. By construction, $\sup_{C_L(v)}Z_u^{(t)}$ and $\sup_{C_L(v')} Z_{u}^{(t)}$ are independent for all non-neighbors $v,v'\in J_z$ (recall that $J_z \subseteq (L\ZZ)^d$), that is, for all $v,v'\in J_z$ such that $\abs{v-v'}> L\sqrt{d}$. Writing $y^*(v)=\sup_{u\in C_L(v)} y_u$ and turning to Lemma~\ref{lem:regvaryingbound}, we find constants $C,K$ such that
\begin{align*}
	\MoveEqLeft
	\sum_{\substack{v<v' \in J_z\\\abs{v-v'}>L\sqrt{d}}} 
	\PP \bigl( 
	\sup_{u\in C_L(v)} (Z_u^{(t)} + y_u)> a_nx,
	\sup_{u\in C_L(v')} (Z_u^{(t)} + y_u)> a_nx \bigr) \\ &
	\le \sum_{v<v' \in J_z} 
	\FF_{L,t}(a_nx - y^*(v)) \FF_{L,t}(a_nx-y^*(v')) \\ &
	\le \bigl(\FF_{L,t}(a_nx)\bigr)^2 
	\Bigl(\sum_{v \in J_z} 
	 C \bigl(K+(y^*(v))_+^\beta \bigr) \Bigr)^2 
\end{align*}
for sufficiently large $n$ and all $z\in \Nk$. Since $\abs{C_n}/k \sim \abs{J_z}$ as $n\to\infty$, we conclude by Theorem~\ref{thm:mixingY} and \eqref{eq:Ztconvergencexn} that
\begin{align*}
	\MoveEqLeft
	\limsup_{n\to\infty }\,\bigl(\FF_{L,t}(a_nx)\bigr)^2 
	\Bigl(\sum_{v \in J_z} 
	 C \bigl(K+(y^*(v))_+^\beta \bigr) \Bigr)^2 \\ & 
	=  \frac{1}{k^2}
	\Bigl( \taut C\bigl(K+\EE \bigl(\sup_{v\in C_L(0)} Y_v \bigr)_+^\beta \bigr) \Bigr)^2 
\end{align*}
almost surely. This is independent of $z$ and of order $o(k^{-1})$ as $k\to\infty$ for all $L\in\NN$. This shows \eqref{eq:anticlusteringfunction} and \eqref{eq:anticlusteringZy} for the terms in the sum with indices more than $L \sqrt{d}$ apart.

For notational convenience, define
\[
	R_L = \int_\Rd \sup_{v\in C_L(0)} \bigl(\fft \bigr)^{\alpha} \dd u .
\]
Now consider $v,v'\in J_z$ such that $|v-v'|\le L \sqrt{d}$. Due to Lemma~\ref{lem:Zrhoequivalence} we have for all fixed $v\neq v'$ that 
\begin{align*}
	\MoveEqLeft
	\lim_{n\to\infty} \frac{1}{\FF_{L,t}(a_nx)}
	\PP \bigl(\sup_{u\in C_L(v) } Z_u^{(t)} > a_nx-y, 
	\sup_{u\in C_L(v') } Z_u^{(t)} > a_nx-y\bigr) \\ &
	\to 
	2 - R_L^{-1}\int_{\Rd} \sup_{s\in C_L(v)\cup C_L(v')} \bigl(f^{(t)}(s-u) \bigr)^{\alpha} \dd u
\end{align*}
uniformly for $\abs{y} \le N$, for all $N\in\NN$. Define $y^*(v,v')=\max\{y^*(v), y^*(v')\}$ and similarly for $(Y_v)_v$ and note that
\begin{equation}\label{eq:Ydoubleexpectation}
		\EE (Y^*(v,v'))_+^\beta
	\le 2\, \EE (\sup_{v\in C_L(0)} Y_v)_+^\beta < \infty .
\end{equation}
Combining the uniform convergence above with Theorem~\ref{thm:mixingY}, Lemma~\ref{lem:regvaryingbound}, equation \eqref{eq:Ztconvergencexn} and the fact that $\abs{C_n}/k \sim \abs{J_z}$ then yield
\begin{align*}
	\MoveEqLeft
	\limsup_{n\to\infty}
	\sum_{\substack{v<v' \in J_z\\\abs{v-v'}\le L \sqrt{d}}} 
	\PP \bigl( 
	\sup_{u\in C_L(v)} (Z_u^{(t)} + y_u)> a_nx,
	\sup_{u\in C_L(v')} (Z_u^{(t)} + y_u)> a_nx \bigr) \\ &
	\le	
	\limsup_{n\to\infty}
	\sum_{\substack{v<v' \in J_z\\\abs{v-v'}\le L\sqrt{d}}} 
	\PP \bigl( 
	\sup_{u\in C_L(v)} Z_u^{(t)}> a_nx - y^*(v,v'),
	\sup_{u\in C_L(v')} Z_u^{(t)}> a_nx - y^*(v,v') \bigr) \\ &	
	\le 
	\frac{\taut}{k} \sum_{\substack{v\in (L\ZZ)^d \\ 0<\abs{v}\le L\sqrt{d}}}
	\EE\bigl(\I{\abs{Y^*(0,v)}\le N} \bigr) \Bigl(2 - R_L^{-1}\int_{\Rd} \sup_{s\in C_L(v)\cup C_L(0)} \bigl(f^{(t)}(s-u) \bigr)^{\alpha} \dd u \Bigr) \\ &
	\quad 
	+ \frac{\taut}{k} \sum_{\substack{v\in (L\ZZ)^d \\ 0<\abs{v}\le L\sqrt{d}}}
	\EE \Bigl(C\bigl(K+ Y^*(0,v)_+^\beta \bigr) \,
	\I{\abs{Y^*(0,v)} > N} \Bigr)
\end{align*}
for all $N\in\NN$. Since this is independent of $z$, and due to the fact that there are only finitely many terms in the sum, we conclude \eqref{eq:anticlusteringZy} by letting $N \to \infty$ using a dominated convergence argument justified by \eqref{eq:Ydoubleexpectation}:
\begin{align*}
	\MoveEqLeft
	\limsup_{n\to\infty}
	\sum_{\substack{v<v' \in J_z\\\abs{v-v'}\le L \sqrt{d}}} 
	\PP \bigl( 
	\sup_{u\in C_L(v)} (Z_u^{(t)} + y_u)> a_nx,
	\sup_{u\in C_L(v')} (Z_u^{(t)} + y_u)> a_nx \bigr) \\ &
	\le	
	\frac{\taut}{k} \sum_{\substack{v\in (L\ZZ)^d \\ 0<\abs{v}\le L\sqrt{d}}}
	\Bigl(2 - R_L^{-1}\int_{\Rd} \sup_{s\in C_L(v)\cup C_L(0)} \bigl(f^{(t)}(s-u) \bigr)^{\alpha} \dd u \Bigr) .
\end{align*}
The proof is finished once we show that this upper bound satisfies \eqref{eq:anticlusteringfunction}. Since $\taut$ is asymptotically equivalent to $L^d$ by \eqref{eq:tauLconvergence}, this is the case if only the ($k$-independent) sum is of order $o(1)$ as $L\to\infty$. As the number of terms in the sum is fixed, it is a matter of showing that the summands tend to $0$. As $f(0)=1$ and $f\ge 0$ it is easily seen that
\[
	2L^d 
	\le \int_{\Rd} \sup_{s\in C_L(v)\cup C_L(0)} \bigl(f^{(t)}(s-u) \bigr)^{\alpha} \dd u
	\le 2 R_L 
\]
for all $0 \neq v$,
and the claim follows since $R_L \sim L^d$ by Lemma~\ref{lem:Lintegralequivalence}.
\end{proof}

\begin{proof}[{\bf Proof of Theorem~\ref{thm:extremelevy}}]
We continue with the notation used throughout this section. 
In particular, 
\[
	\taut = x^{-\alpha} \rho((1,\infty)) \int_{\Rd} \sup_{v\in C_L(0)} \bigl(\fft\bigr)^\alpha \dd u 
\]
for all fixed $x>0$.
Now consider $J_z$ for $z\in\Nk$. Recalling the definition of $\myL$ and $\myt$ and the fact that $\myt \le \myL$, it can easily be seen that
\begin{align*}
	\MoveEqLeft
	\sum_{v\in J_z} \PP(\myt(\{v\}) > a_nx) 
	- \sum_{v<v'\in J_z}\PP(\myt(\{v\}) > a_nx, \myt(\{v'\})> a_nx) \\ &
	\le  \PP(\myL(J_z) > a_nx) \\ &
	\le \sum_{v\in J_z} \PP(\myL(\{v\}) > a_nx) .
\end{align*}
We now obtain from Lemma~\ref{lem:supBnbound} that
\begin{equation}\label{eq:supBnbounds}
\begin{aligned}
	\MoveEqLeft	
	\Bigl(\liminf_{n\to\infty} \min_{z\in \Nk}\Bigl(1-\sum_{v\in J_z} \PP\bigl(\sup_{u\in C_L(v)}(Z_u + y_u) > a_nx \bigr)\Bigr)\Bigr)^{\tilde q_{k,L}} \\&
	\le \liminf_{n\to\infty} \PP\bigl( \sup_{v\in C_n} (Z_v + y_v) \le a_nx \bigr) \\ &
	\le \limsup_{n\to\infty} \PP\bigl( \sup_{v\in C_n} (Z_v + y_v) \le a_nx \bigr) \\ &
	\le \Bigl(\limsup_{n\to\infty} \max_{z\in \Nk}\Bigl(1-\sum_{v\in J_z} \PP\bigl(\sup_{u\in C_L(v)}(Z_u^{(t)} + y_u) > a_nx \bigr) + S_{n,k,L}(z)\Bigr)\Bigr)^{\tilde p_{k,L}},	
\end{aligned}
\end{equation}
where 
\[
	S_{n,k,L}(z) = \sum_{v<v' \in J_z} 
	\PP \bigl( 
	\sup_{u\in C_L(v)} (Z_u^{(t)} + y_u)> a_nx,
	\sup_{u\in C_L(v')} (Z_u^{(t)} + y_u)> a_nx
	\bigr),
\]
and $\tilde q_{k,L}=\limsup_n \qq$ and 
$\tilde p_{k,L}=\liminf_n \pp$. By Lemma~\ref{lem:anticlustering} there is a sequence of functions $g_L(k)$ satisfying 
\begin{equation*}
	\limsup_{k\to\infty} k \,g_L(k) = o(L^d)
\end{equation*} 
as $L\to\infty$, such that $\limsup_n S_{n,k,L}(z) \le g_L(k) $ uniformly in $z$. Since $\abs{C_n}/k\sim \abs{J_z}$, we find by Lemma~\ref{thm:partialy} and
\eqref{eq:supBnbounds} that
\begin{align*}
	\Bigl(1-\frac{\tau_L}{k}\Bigr)^{\tilde q_{k,L}}&
	\le \liminf_{n\to\infty} \PP\bigl( \sup_{v\in C_n} (Z_v + y_v) \le a_nx \bigr) \\ &
	\le \limsup_{n\to\infty} \PP\bigl( \sup_{v\in C_n} (Z_v + y_v) \le a_nx \bigr) \\ &
	\le \Bigl(1-\frac{\taut}{k} + g_L(k)\Bigr)^{\tilde p_{k,L}} 
\end{align*}
almost surely.
Since
\[
	\tilde p_{k,L} \sim \tilde q_{k,L} \sim \frac{k}{L^d} 
\]
as $k\to\infty$, we use the equivalence $\log(1-y)\sim -y$ ($y\to 0$) to obtain, as $k\to\infty$,
\begin{align*}
	\exp \Bigl(-\frac{\tau_L}{L^d}\Bigr)&
	\le \liminf_{n\to\infty} \PP\bigl( \sup_{v\in C_n} (Z_v + y_v) \le a_nx \bigr) \\ &
	\le \limsup_{n\to\infty} \PP\bigl( \sup_{v\in C_n} (Z_v + y_v) \le a_nx \bigr) \\ &
	\le \exp\Bigl(-\frac{\taut}{L^d} + o(1)\Bigr) .
\end{align*}
Here the remainder $o(1)$ is with respect to the limit $L\to\infty$.
Secondly, taking the limit $L\to\infty$ and using the equivalence \eqref{eq:tauLconvergence} show that  
\begin{equation*}
	\lim_{n\to\infty} \PP\bigl( \sup_{v\in C_n} (Z_v + y_v) \le a_nx \bigr) 
	= \exp \Bigl(-x^{-\alpha} \rho((1,\infty)) \Bigr) 
\end{equation*}
almost surely.

Let $\pi$ denote the distribution of the field $(Y_v)_v$. Then, by independence and dominated convergence,
\begin{align*}
	\PP\bigl( \sup_{v\in C_n} X_v \le a_nx \bigr) &
	= \int \PP\bigl( \sup_{v\in C_n} (Z_v + y_v) \le a_nx \bigr) \pi (\dd y)
	\to \exp \Bigl(-x^{-\alpha} \rho((1,\infty)) \Bigr) 
\end{align*}
as $n\to\infty$.
\end{proof}

\begin{proof}[{\bf Proof of Theorem~\ref{thm:extremelevyv2}}]
Following the proofs in this section, it can easily be seen that the claim in Theorem~\ref{thm:extremelevy} can be extended as follows: 
\begin{equation}\label{eq:extremeextended}
	\PP\Bigl(  
	\sup_{v\in C_n} \bigl( X_v + Y_v^1\bigr)
	\le a_n x \pm d_n \Bigr)
	\to \exp\Bigl( - x^{-\alpha}\rho((1,\infty)) \Bigr)
\end{equation}
for all $x>0$, where $(Y_v^1)_v$ is the sufficiently light-tailed and ergodic field given by \eqref{eq:Y1integrability} in the present theorem, and $d_n$ is a sequence of order $o(a_n)$ tending to $\infty$.

Consider the stationary field $(Y_v^2)_v$ and note that the integrability assumed in \eqref{eq:Y2integrability} is equivalent to
\begin{equation*}
	\begin{aligned}
	&
	\EE \bigl(\sup_{v\in C_1(0)} Y_v^2 \bigr)_+^{\beta_2} < \infty ,
	\qquad \text{and} \\&
	\EE \bigl(\inf_{v\in C_1(0)} Y_v^2 \bigr)_-^{\beta_2} < \infty ,
	\end{aligned}
\end{equation*}
where $y_-^{\beta_2} = (-y  \I{y \le 0})^{\beta_2}$, and $C_1(0)$ denotes the unit cube in $\Rd$. With this $\beta_2>\alpha$ in mind, choose the sequence $d_n$ according to Lemma~\ref{lem:regvarloworder}. Then, due to Lemma~\ref{lem:regvarloworder} and the fact that $d_n\to \infty$,
\begin{equation}\label{eq:0limitsup}
	\abs{C_n} \PP\bigl( \sup_{v\in C_1(0)} Y_v^2 > d_n\bigr)
	\sim
	\frac{\rho((1,\infty)) 
	\PP\bigl( (\sup_{v\in C_1(0)} Y_v^2)_+ > d_n\bigr)}{\rho((a_n,\infty))}
	\to 0
\end{equation}
as $n\to\infty$, and similarly
\begin{equation}\label{eq:0limitinf}
	\abs{C_n} \PP\bigl( \inf_{v\in C_1(0)} Y_v^2 < -d_n\bigr)
	\sim
	\frac{\rho((1,\infty)) 
	\PP\bigl( (\inf_{v\in C_1(0)} Y_v^2)_- > d_n\bigr)}{\rho((a_n,\infty))}
	\to 0 .
\end{equation}
Now fix $k=L=1$ and recall the discrete set $K_n=K_{n,1} \subseteq \Zd$ from Theorem~\ref{thm:maingeometrictheorem}, which in particular satisfies that
\[
	C_n \subseteq \bigcup_{z\in K_n} C_1(z)
\]
and $\abs{K_n} \le c \cdot \abs{C_n}$ for some finite, $n$-independent $c$. Turning to \eqref{eq:0limitsup} and the stationarity of $(Y_v^2)$ we thus see that
\begin{equation}\label{eq:0limitsup1}
	\begin{aligned}
	\PP\bigl( \sup_{v\in C_n} Y_v^2 > d_n \bigr) &
	\le
	\PP\bigl( \max_{z\in K_n} \sup_{v\in C_1(z)} Y_v^2 > d_n \bigr) \\&
	\le
	c\cdot \abs{C_n}
	\PP\bigl( \sup_{v\in C_1(0)} Y_v^2 > d_n \bigr)
	\to 0
	\end{aligned}
\end{equation}
as $n\to\infty$. Similarly, by \eqref{eq:0limitinf},
\begin{equation}\label{eq:0limitinf1}
	\PP\bigl( \inf_{v\in C_n} Y_v^2 < -d_n \bigr)
	\to 0
\end{equation}
as $n\to\infty$. Since, by independence of the fields involved,
\begin{align*}
	\MoveEqLeft	
	\PP\bigl( \sup_{v\in C_n}(X_v + Y_v^1) 
	\le a_n x - d_n \bigr)
	\PP \bigl(\sup_{v\in C_n} Y_v^2 \le d_n \bigr) 
	\\ &
	\le
	\PP\bigl( \sup_{v\in C_n}(X_v + Y_v^1 + Y_v^2) 
	\le a_n x \bigr) \\ &
	\le 
	\PP\bigl( \sup_{v\in C_n}(X_v + Y_v^1) 
	\le a_n x + d_n \bigr)
	\PP \bigl(\inf_{v\in C_n} Y_v^2 \ge -d_n \bigr)
	+ \PP \bigl(\inf_{v\in C_n} Y_v^2 < -d_n \bigr) ,
\end{align*}
the claim follows by \eqref{eq:extremeextended} in combination with \eqref{eq:0limitsup1} and \eqref{eq:0limitinf1}.
\end{proof}



  \bibliographystyle{elsarticle-harv} 
  \bibliography{bibref.bbl}


%
\end{document}